\documentclass[11pt,a4paper]{article}
\usepackage[T1]{fontenc}
\usepackage[utf8]{inputenc}
\usepackage{lmodern}
\usepackage{amsmath,amssymb,amsthm,mathtools}
\usepackage{mathrsfs}
\usepackage[margin=1in]{geometry}
\usepackage{microtype,booktabs,array,enumitem}
\usepackage[hidelinks]{hyperref}
\usepackage{indentfirst}
\allowdisplaybreaks[2]
\newtheorem{maintheorem}{Theorem}
\newtheorem{maincorollary}{Corollary}[section]

\theoremstyle{plain}
\newtheorem{theorem}{Theorem}[section]
\newtheorem{lemma}[theorem]{Lemma}
\newtheorem{proposition}[theorem]{Proposition}

\theoremstyle{definition}

\theoremstyle{remark}
\newtheorem{remark}[theorem]{Remark}
\numberwithin{equation}{section}
\DeclareMathOperator{\meas}{meas}

\begin{document}
\title{\bfseries Better than square-root cancellation and Gaussian behavior on Piatetski-Shapiro Sequences}
\author{Renjie Zhu\thanks{Address: School of Mathematics and Statistics, Shaanxi Normal University, Xi'an, 710119, Shaanxi, P. R. China. Email: zhurenjie@snnu.edu.cn},  Tianping Zhang\thanks{
Tianping Zhang is the corresponding author. 
Address 1: School of Mathematics and Statistics, Shaanxi Normal University, Xi'an, 710119, Shaanxi, P. R. China. 
Address 2: Research Center for Number Theory and Its Applications, Northwest University, Xi'an, 710127, Shaanxi, P. R. China. 
Email: tpzhang@snnu.edu.cn
}}
\date{}
\maketitle
\begin{abstract}
In this paper, we investigate whether Harper's better than square-root cancellation phenomenon still appears in Piatetski-Shapiro sequences, and when one should expect only square-root size. If $c=c(X)$ satisfies $(c-1)\log X\to 0$ as $X\to\infty$, then for Steinhaus or Rademacher random multiplicative functions we have
\[
\mathbb{E}\left|\sum_{n\in\mathcal{N}^{(c)}(X)}f(n)\right|=o\!\left(\sqrt{|\mathcal{N}^{(c)}(X)|}\right).
\]

For Lebesgue almost every fixed $1<c<\frac{8}{7}$, we prove the multiplicative energy satisfies
\[
E^\times(\mathcal{N}^{(c)}(X))=2|\mathcal{N}^{(c)}(X)|^2-|\mathcal{N}^{(c)}(X)|+o\left(|\mathcal{N}^{(c)}(X)|^2\right).
\]

Together with a largest prime factor estimate, this allows us to apply the criterion of Soundararajan and Xu and obtain a standard complex Gaussian limit for the normalized Steinhaus sum. To the best of our knowledge, this gives the first asymptotic formula for the multiplicative energy of Piatetski-Shapiro sequences. For fixed exponent argument, our proof combines metric arguments with the Van der Corput's $B$-process and a double large sieve estimate with Robert-Sargos spacing theorem.
\end{abstract}


\section{Introduction}
\subsection{Background}

\indent Square-root cancellation, a fundamental phenomenon in number theory, particularly in analytic number theory, attracts intensive attention and is associated with many important problems and conjectures. For example, the Riemann Hypothesis is equivalent to the estimate
\[
M_{\mu}(x):=\sum_{n\le x}\mu(n)\ll_{\varepsilon}x^{\frac12+\varepsilon}.
\]
The classical P\'olya-Vinogradov inequality also gives a bound at the square-root scale up to a logarithmic factor,
\[
M_{\chi}(q):=\max_{M,N}\left|\sum_{n=M+1}^{M+N}\chi(n)\right|\ll \sqrt q\log q,
\]
where $\chi$ is a non-principal Dirichlet character modulo $q$. For standard background on these classical estimates, we refer to
\cite{IwaniecKowalski}.

It is usually believed that the square-root scale is a key barrier which is very difficult to overcome in many problems. However, an important exception comes from Harper's remarkable resolution \cite{HarperRMF} of Helson's conjecture \cite{Helson}. Define a Steinhaus random multiplicative function as a completely multiplicative function supported on positive integers such that $f(p)$ takes values independently and uniformly on the complex unit circle for all primes $p$. A Rademacher random multiplicative function is supported on square-free integers and is defined by taking the values $f(p)$ independently equal to $\pm1$ with probability $1/2$ and setting
\[
f(n)=\prod_{p\mid n}f(p)
\]
for square-free $n$, and $f(n)=0$ otherwise. Harper \cite{HarperRMF} shows that the partial sums of random multiplicative functions have better than square-root cancellation
\begin{equation}
\mathbb{E}\left|\sum_{n\le x}f(n)\right|\asymp\frac{\sqrt x}{(\log\log x)^{1/4}},
\label{eq:Harper-rmf}
\end{equation}
whenever $f$ is of Steinhaus or Rademacher type.

Another important result on character sums was also established by Harper \cite{HarperChar}. A Dirichlet character modulo a sufficiently large prime behaves in several aspects like a random multiplicative function. By utilizing Steinhaus random multiplicative functions to model characters, Harper proved that typical partial character sums also have better than square-root cancellation,
\begin{equation}
\frac{1}{p-1}\sum_{\chi\bmod p}\left|\sum_{n\le x}\chi(n)\right|\ll\frac{\sqrt x}{(\log\log(10L))^{1/4}},\qquad L=\min\{x,p/x\},
\label{eq:Harper-character}
\end{equation}
where $p$ is a large prime. For further related work, it can be seen in \cite{HarperShort,WangXuBeyond}.

These results encourage us to ask whether better than square-root cancellation can still appear when the interval of summation is replaced by a more general sparse set. A different behavior was established by Soundararajan and Xu [19]. They gave criteria for Steinhaus random multiplicative sums over deterministic sets to have asymptotically complex Gaussian behavior. In particular, asymptotically minimal multiplicative energy is the main arithmetic condition in their simplified criterion. Under their hypotheses, one has
\[
\frac{1}{\sqrt{|\mathcal{A}|}}\sum_{n\in\mathcal{A}}f(n)\xrightarrow{d}\mathcal{CN}(0,1).
\]
This suggests studying the multiplicative energy of the summation set, together with the additional concentration conditions required in the sparse setting.

Nevertheless, better than square-root cancellation can still occur on sparse sets with substantial multiplicative structure. The first positive example in this direction was given by Xu \cite{XuRough}. If we define the rough number set as
\[
\mathcal{A}_R(x):=\{n\le x:\ p\mid n\Longrightarrow p\ge R\},
\]
then, in a suitable range of $R$, we have better than square-root cancellation for the partial sum
\begin{equation}
\mathbb{E}\left|\sum_{n\in\mathcal{A}_R(x)}f(n)\right|=o\!\left(\sqrt{|\mathcal{A}_R(x)|}\right).
\label{eq:rough}
\end{equation}
\indent Recently, an analogous result for the smooth number set
\[
\mathcal{A}_y(x):=\{n\le x:\ p\mid n\Longrightarrow p\le y\}
\]
was obtained by Hardy and Xu \cite{HardyXu}:
\begin{equation}
\mathbb{E}\left|\sum_{n\in\mathcal{A}_y(x)}f(n)\right|=o\!\left(\sqrt{|A_y(x)|}\right).
\label{eq:smooth}
\end{equation}
\indent As for a general set $\mathcal A$, Xu \cite{XuRough} points out that a complete characterization seems difficult because the relevant multiplicative structure cannot be described in a simple way. This inspires us to search for other natural sparse sets on which both better than square-root cancellation and square-root Gaussian behavior can be investigated.

A classical example of a sparse sequence is the Piatetski-Shapiro sequence
\[
\mathcal{N}^{(c)}:=\bigl(\lfloor n^c\rfloor\bigr)_{n=1}^{\infty},\qquad c>1,
\]
where we also denote the truncation at $X$ as $\mathcal{N}^{(c)}(X)=\mathcal{N}^{(c)}\cap[1,X]$. Piatetski-Shapiro \cite{PiatetskiShapiro} first proved that this sequence contains infinitely many primes in a certain range of $c$. Later many classical problems in analytic number theory were studied on Piatetski-Shapiro sequences, including prime distribution and its interaction with Beatty sequences \cite{BanksShparlinskiBeatty,GuoBeatty}, exponential sums \cite{LuGuoGuo}, and character sums \cite{BakerBanks,BanksShparlinskiChar}.

Piatetski-Shapiro sequences are especially natural for the problem considered here because the same family contains two rather different regimes. If $c=c(X)$ approaches $1$, then $\mathcal{N}^{(c)}(X)$ is sufficiently close to the full interval $[1,X]$. Therefore, it is natural to expect that Harper's better than square-root cancellation for ordinary random multiplicative partial sums may still be preserved. If $c>1$ is fixed, however, then
\[
|\mathcal{N}^{(c)}(X)|\asymp X^{1/c}=o(X),
\]
so the Piatetski-Shapiro set is now genuinely sparse in $[1,X]$. Thus the preceding comparison with the full interval cannot be used. This naturally leads us to study the multiplicative energy of $\mathcal{N}^{(c)}(X)$.

Our results show that these two regimes indeed have different behaviors. When $c=c(X)$ is sufficiently close to $1$, we prove that better than square-root cancellation is preserved. For almost every fixed $1<c<8/7$,
we instead prove that the multiplicative energy of the truncated Piatetski-Shapiro sequence is asymptotically minimal,
\[
E^{\times}\!\left(\mathcal{N}^{(c)}(X)\right)=2|\mathcal{N}^{(c)}(X)|^2-|\mathcal{N}^{(c)}(X)|+o\!\left(|\mathcal{N}^{(c)}(X)|^2\right),
\]
where $E^\times(A)=\#\{(m_1,m_2,n_1,n_2)\in A^4:m_1m_2=n_1n_2\}$. Combining this with the criterion of Soundararajan and Xu \cite{SoundXu}, we obtain a standard complex Gaussian limit for the corresponding Steinhaus sum. To the best of our knowledge, this gives the first asymptotic evaluation of the multiplicative energy of Piatetski-Shapiro sequences.

Thus better than square-root cancellation occurs when $c$ is close to $1$, while for almost every fixed $1<c<8/7$ the corresponding sum has square-root Gaussian behavior. We leave the transition between these two cases open.

\subsection{Main results}

For a real number $X\ge3$, we write
\[
S_c(X;f)=\sum_{n\in\mathcal{N}^{(c)}(X)}f(n),\qquad S_X(f)=\sum_{n\le X}f(n),
\]
where $c=c(X)$ is deterministic whenever it varies with $X$. Our first result concerns the case $c$ tends to $1$.

\begin{maintheorem}\label{thm:near}
Let $f$ be either a Steinhaus or a Rademacher random multiplicative function. If $c=c(X)>1$ satisfies
\[
(c-1)\log X\longrightarrow0,
\]
then we have
\begin{equation}\label{eq:mainnear}
\mathbb{E}|S_c(X;f)|=o\!\left(\sqrt{|\mathcal{N}^{(c)}(X)|}\right).
\end{equation}
\end{maintheorem}

Theorem~\ref{thm:near} shows that Harper's better than square-root phenomenon persists when $c$ approaches $1$ sufficiently fast. Under a slightly stronger assumption, the perturbation is small enough so that the first absolute moment itself is asymptotic to that of the ordinary partial sum. This gives the following refinement.

\begin{maincorollary}\label{cor:comparison}
Let $f$ be either a Steinhaus or a Rademacher random multiplicative function. If
\[
(c-1)\log X\sqrt{\log\log X}\longrightarrow0,
\]
then we have
\begin{equation}\label{eq:maincomparison}
\mathbb{E}|S_c(X;f)|=(1+o(1))\mathbb{E}|S_X(f)|\asymp\frac{\sqrt X}{(\log\log X)^{1/4}}.
\end{equation}
\end{maincorollary}

The argument behind Theorem~\ref{thm:near} also gives an arithmetic consequence. Harper \cite[Theorem~1]{HarperChar} proved
better than square-root cancellation for typical character sums over an interval. Applying the same comparison to his estimate gives the following consequence.
\begin{maincorollary}\label{thm:characters}
Let $X\to\infty$, and $p=p(X)$ be prime with
\[
X<p,\qquad \frac{p}{X}\longrightarrow\infty.
\]
If $c=c(X)>1$ satisfy
\[
(c-1)\log X\longrightarrow0,
\]
then we have
\begin{equation}\label{eq:charmain}
\frac{1}{p-1}\sum_{\chi\bmod p}\left|\sum_{n\in\mathcal{N}^{(c)}(X)}\chi(n)\right|=o\!\left(\sqrt{|\mathcal{N}^{(c)}(X)|}\right).
\end{equation}
\end{maincorollary}

We now turn to fixed exponents, where the situation is quite different. The comparison with the full interval used above is no longer available, and we are led instead to study the multiplicative structure of $\mathcal{N}^{(c)}(X)$. The next theorem shows that, for almost every fixed $c$ in a substantial range, its multiplicative energy is asymptotically minimal.

\begin{maintheorem}\label{thm:energy}
For Lebesgue almost every fixed $c\in(1,8/7)$, we have the following asymptotic formula
\begin{equation}\label{eq:mainenergy}
E^\times\!\left(\mathcal{N}^{(c)}(X)\right)=2|\mathcal{N}^{(c)}(X)|^2-|\mathcal{N}^{(c)}(X)|+o\!\left(|\mathcal{N}^{(c)}(X)|^2\right),
\quad as\ X\to\infty.
\end{equation}
\end{maintheorem}

Thus, for almost every fixed exponent in this range, the two diagonal families give the whole main term, and all non-diagonal solutions contribute only $o(|\mathcal{N}^{(c)}(X)|^2)$. The qualifier ``almost every'' in Theorem~\ref{thm:energy} cannot be removed. In fact, there is an explicit sequence of exceptional exponents tending to $1$.
\begin{proposition}\label{prop:exceptional}
For every fixed integer $u\ge5$, put $c_u:=\frac{\log(u+1)}{\log u}$. Then we have $1<c_u<\frac{8}{7}$ and
\begin{equation}\label{eq:exceptionalfamily}
\liminf_{X\to\infty}\frac{E^\times(\mathcal{N}^{(c_u)}(X))}{|\mathcal{N}^{(c_u)}(X)|^2}\ge2+\frac{2}{u^2(u+1)^2}>2.
\end{equation}
\end{proposition}

Hence the multiplicative energy in Theorem~\ref{thm:energy} is not asymptotically minimal for every fixed $c\in(1,8/7)$. Combining Theorem~\ref{thm:energy} with a largest prime factor estimate, the criterion of Soundararajan and Xu \cite[Corollary~3.2]{SoundXu} gives the following Gaussian limit.

\begin{maintheorem}\label{thm:gaussian}
Let $f$ be a Steinhaus random multiplicative function. For the same full-measure set of fixed $c\in(1,8/7)$ as in Theorem~\ref{thm:energy}, we have the following central limit theorem
\begin{equation}\label{eq:mainclt}
\frac{S_c(X;f)}{\sqrt{|\mathcal{N}^{(c)}(X)|}}\xrightarrow{d}\mathcal{CN}(0,1).
\end{equation}
Moreover,
\begin{equation}\label{eq:maincltfirst}
\mathbb{E}|S_c(X;f)|\sim\frac{\sqrt\pi}{2}\sqrt{|\mathcal{N}^{(c)}(X)|}.
\end{equation}
\end{maintheorem}

Here $\mathcal{CN}(0,1)$ denotes the standard complex Gaussian distribution with density $\pi^{-1}e^{-|z|^2}$ on $\mathbb{C}$.

\begin{remark}\label{rem:87intro}
The endpoint $8/7$ is a limitation of the present method. Writing $\gamma=1/c$, the condition $\gamma>7/8$ appears independently in the second-derivative argument and in the argument using the Van der Corput $B$-process and the double large sieve. Meanwhile, the characterization of transition point between better than square-root cancellation and Gaussian behavior on Piatetski-Shapiro sequences is still open.
\end{remark}

\subsection{Outline of the proof}

We briefly describe the main ideas of the proofs. The argument for Theorem~\ref{thm:near} is a direct comparison with the ordinary partial sum, whereas the proof of Theorem~\ref{thm:energy} requires a more detailed analysis of the multiplicative relations in $\mathcal{N}^{(c)}(X)$.

For Theorem~\ref{thm:near}, we define
\[
D_c(X):=([1,X]\cap\mathbb{N})\setminus\mathcal{N}^{(c)}(X).
\]
The orthogonality of Steinhaus or Rademarcher random multiplicative functions can give the bound
\[
\|S_X(f)-S_c(X;f)\|_2\le \sqrt{|D_c(X)|}.
\]
Since
\[
|D_c(X)|=X-X^{1/c}+O(1)=X\left(1-e^{-\frac{c-1}{c}\log X}\right)+O(1),
\]
the condition $(c-1)\log X\to0$ implies $|D_c(X)|=o(X)$. Theorem~\ref{thm:near} then follows from Harper's estimate \cite[Theorems~1 and~2]{HarperRMF} for the ordinary random multiplicative partial sum directly. The two corollaries are obtained from the same comparison, using Harper's first-moment estimate for character sums \cite[Theorem~1]{HarperChar} in the arithmetic case.

We now turn to Theorem~\ref{thm:energy}. Put $\gamma=1/c$ and restrict the problem to a dyadic block $m_1,m_2,n_1,n_2\asymp Y$. Every solution of
\[
m_1m_2=n_1n_2
\]
has a unique representation
\[
m_1=ga,\qquad m_2=hb,\qquad n_1=gb,\qquad n_2=ha,\qquad (a,b)=1.
\]
The two diagonal families correspond to the case $a=b$ or $g=h$. Thus it remains to count the solutions with $a\ne b$ and $g\ne h$.

After a dyadic decomposition, and after interchanging the two pairs of variables if necessary, we may suppose that
\[
a,b\asymp A,\qquad g,h\asymp G,\qquad AG\asymp Y,\qquad A\le G,
\]
with a loss of at most $Y^{o(1)}$.

We next approximate the indicator function of the Piatetski-Shapiro set by a smooth periodic function and expand it into a Fourier series. After multiplying the four truncated Fourier expansions and fixing $a,b$ together with the frequencies $r_1,s_1,r_2,s_2$, we are led to exponential sums of the form
\[
\sum_{g\asymp G}e(C_1g^\gamma),\qquad\sum_{h\asymp G}e(C_2h^\gamma),
\]
where
\[
C_1=r_1a^\gamma+s_1b^\gamma,\qquad C_2=r_2a^\gamma+s_2b^\gamma.
\]
Thus, we put a parameter
\[
W=G^\gamma|C|
\]
for such a sum. Next we need to divide the argument according to the size of $W$.

The first case comes from the classical Van der Corput second-derivative estimate. When $W$ lies in an intermediate range $Y^\eta<W\le G$ for a suitable $\eta>2(1-\gamma)$, Van der Corput's result
\[
\sum_{g\asymp G} e(Cg^\gamma)\ll GW^{-1/2}+W^{1/2}
\]
suffices to give the required saving.

Therefore, it remains to treat separately the two ranges in which this estimate is not effective, that is,
\[
0<W\le Y^\eta\qquad\text{and}\qquad W>G.
\]

We now consider the case $0<W\le Y^\eta$. Since $W=|C|G^\gamma$, this is precisely the range in which the coefficient $C$ can be very small. Notice that $C$ contains two terms involving non-zero Fourier frequencies, hence $W$ lies in such a range only when the two terms in $C$ have opposite signs and are very close to each other. After a dyadic decomposition of the frequencies, we are reduced to expressions of the form
\[
ra^\gamma-sb^\gamma,
\]
with $r,s>0$. The condition $W\le Y^\eta$ gives
\[
|ra^\gamma-sb^\gamma|\ll A^\gamma Y^{\eta-\gamma}.
\]

The point is that, for fixed $a,b,r,s$, this inequality can hold only for a small set of values of $\gamma$. We can bound the measure of such a set of $\gamma$ by
\[
\ll\frac{\delta}{|\log(a/b)|}\ll\frac{\delta A}{|a-b|},
\]
where $\delta$ denotes the corresponding relative error.

We split the argument according to the size of $A$. Writing $s_0=1-\gamma_0$, we split atPut $A_0=Y^{2s_0+\kappa}$. When $A\le A_0$, summing the preceding measure estimate over $a,b$ and over the Fourier frequencies shows that the size of exceptional sets can be controlled provided
\[
\eta<1-6(1-\gamma)-2\kappa.
\]
The Borel--Cantelli lemma ensures that outside a set of $\gamma$ of measure zero, the contribution from the range $A\le A_0$ is $o(Y^{2\gamma})$. For $A>A_0$, a first-derivative count in $b$ gives the required bound deterministically. Combining this with the requirement given by intermediate range $Y^{\eta}<W<G$, we find the choice of parameter $\eta$ is possible whenever $\gamma>7/8$.

Finally it remains to treat the range $W>G$. We apply the Van der Corput $B$-process to
\[
\sum_{g\asymp G} e(Cg^\gamma).
\]
The transformed sum has length
\[
K\asymp \frac{W}{G},
\]
and, apart from a smooth weight, its exponential factor is of the form
\[
e\!\left(\operatorname{sgn}(C)c_\gamma|C|^{1/(1-\gamma)}|k|^{-\gamma/(1-\gamma)}\right),
\]
where $c_\gamma\ne0$ depends only on $\gamma$.

We then apply the double large sieve to the variables $(a,b)$ and the dual variable $k$. After Cauchy--Schwarz inequality, the required spacing estimate on the $(a,b)$-side is reduced to the four-variable estimate of Robert and Sargos \cite{RobertSargos}. 

When both frequencies in the selected large phase are nonzero, the contribution is less than $Y^{7/4+2\epsilon+o(1)}$. If exactly one frequency in the selected large phase is nonzero, the contribution is less than $Y^{\gamma+7/8+2\epsilon+o(1)}$. Since $\gamma>7/8$, both contributions are $o(Y^{2\gamma})$ provided that $\epsilon>0$ is sufficiently small.

Combining this estimate with the preceding treatment of $W\le G$, we obtain
\[
E^\times(\mathcal{N}^{(c)}(X))=2|\mathcal{N}^{(c)}(X)|^2-|\mathcal{N}^{(c)}(X)|+o\!\left(|\mathcal{N}^{(c)}(X)|^2\right)
\]
for almost every fixed $\gamma\in(7/8,1)$, and hence prove Theorem~\ref{thm:energy}.

The exceptional exponents in Proposition~\ref{prop:exceptional} arise from exact dilation relations. For an integer $u\ge5$, put $v=u+1$ and $c_u=\frac{\log v}{\log u}$ so that $u^{c_u}=v$. If we define
\[
m=\lfloor n^{c_u}\rfloor,\qquad\{n^{c_u}\}<\frac{1}{v},
\]
then it follows that
\[
\lfloor(un)^{c_u}\rfloor=\lfloor vn^{c_u}\rfloor=vm.
\]
The sequence $\{n^{c_u}\}$ is uniformly distributed modulo one, so this relation holds for a positive proportion of $n$. Pairing two distinct such elements produces a positive proportion of additional non-diagonal multiplicative relations. This gives
\[
E^\times(\mathcal{N}^{(c_u)}(X))\ge\left(2+\frac{2}{u^2v^2}+o(1)\right)|\mathcal{N}^{(c_u)}(X)|^2.
\]
Since $v=u+1$, this proves Proposition~\ref{prop:exceptional}.

Theorem~\ref{thm:gaussian} follows from the criterion of Soundararajan and Xu. The energy theorem verifies the multiplicative paucity condition, while the remaining condition can be checked by the standard structure.

\section{Preliminaries}
Throughout the paper, we write $\gamma=\frac{1}{c}$, $e(t)=e^{2\pi it}$ and $\{t\}=t-\lfloor t\rfloor$. The symbol $n\sim N$ means $N<n\le 2N$ and $A\asymp B$ means both $A\ll B$ and $B\ll A$. For $n\ge2$, let $P^{+}(n)$ denote the largest prime factor of $n$.

\begin{lemma}[Piatetski-Shapiro indicator]\label{lem:indicator}
Let $c>1$ and $c\notin \mathbb{N}$. For every integer $m\ge1$,
\[
\chi_{\gamma}(m):=\mathbf{1}_{\mathcal{N}^{(c)}}(m)=\lfloor-m^\gamma\rfloor-\lfloor-(m+1)^\gamma\rfloor.
\]
\end{lemma}

\begin{proof}
This is the standard indicator identity for Piatetski-Shapiro sequences; see, for example, \cite{Kumchev}.
\end{proof}

For later use, put
\[
\Delta_\gamma(m):=(m+1)^\gamma-m^\gamma.
\]
Then Lemma~\ref{lem:indicator} gives the equivalent relation
\[
\mathbf{1}_{\mathcal{N}^{(c)}}(m)=1\quad\Longleftrightarrow\quad\Delta_\gamma(m)>\{-m^\gamma\}.
\]
For $m\asymp Z$ we use the medium value theorem and get 
\[
\Delta_\gamma(m)\asymp_\gamma Z^{\gamma-1}.
\]

To obtain the Fourier expansion needed in the proof of Theorem~\ref{thm:energy}, we majorize this sharp cutoff by the following smooth periodic function.

\begin{lemma}[Smooth periodic majorant]\label{lem:majorant}
Let $I=[\gamma_0,\gamma_1]$ be a compact subinterval of $(0,1)$, $\gamma\in I$, and $Z\ge2$. Put $H=Z^{1-\gamma}$. Then there exists a non-negative smooth $1$-periodic function $\Psi=\Psi_{Z,\gamma}$ such that
\[
\chi_\gamma(m)\le \Psi(-m^\gamma)\qquad\left(\frac Z4\le m\le16Z\right),
\]
and
\[
\Psi(t)=\sum_{r\in\mathbb{Z}}c_r e(rt),
\]
where for every $B>0$ it follows
\begin{equation}\label{eq:majorant-coeff}
c_r\ll_{I,B}H^{-1}\left(1+\frac{|r|}{H}\right)^{-B},\qquad\sum_{r\in\mathbb Z}|c_r|\ll_I1.
\end{equation}
Meanwhile, for every fixed $\varepsilon>0$ and $A>0$, we have
\begin{equation}\label{eq:majorant-tail}
\sum_{|r|>HY^\varepsilon}|c_r|\ll_{I,\varepsilon,A}Y^{-A}\qquad (Y\ge2).
\end{equation}
\end{lemma}

\begin{proof}
For $Z/4\le m\le16Z$, the mean value theorem gives
\[
\Delta_\gamma(m)\ll_I H^{-1}.
\]
Choose a fixed non-negative function $\phi\in C_c^\infty(\mathbb R)$ such that $\phi(u)\ge1$ for $0\le u\le C_I$, where $C_I$ is sufficiently large. Define a new function as
\[
\Psi(t)=\sum_{j\in\mathbb{Z}}\phi\bigl(H(t+j)\bigr).
\]
It can be verified the function $\Psi$ is smooth and $1$-periodic. Hence it has the Fourier expansion
\[
\Psi(t)=\sum_{r\in\mathbb{Z}}c_r e(rt),
\]
where $c_r=\int_0^1\Psi(t)e(-rt)\,dt$. By the definition of $\Psi$ and unfolding the integral, we know that
\begin{align*}
c_r&=\sum_{j\in\mathbb{Z}}\int_0^1\phi\bigl(H(t+j)\bigr)e(-rt)\,dt  \\
&=\int_{\mathbb{R}}\phi(Hu)e(-ru)\,du=H^{-1}\widehat\phi(r/H).
\end{align*}

Since $\widehat\phi$ is rapidly decreasing, it can be deduced
\[
|c_r|\ll_{I,B}H^{-1}\left(1+\frac{|r|}{H}\right)^{-B}.
\]
Summing over $r$ gives the second estimate in \eqref{eq:majorant-coeff} directly. While the inequality \eqref{eq:majorant-tail}
holds if we take $B$ sufficiently large.
\end{proof}

Next we introduce the following elementary form of the Borel--Cantelli lemma which will be used later.

\begin{lemma}[First Borel--Cantelli lemma]\label{lem:borel-cantelli}
Let $(E_n)_{n\ge1}$ be a sequence of measurable subsets of a measure space $(\Omega,\mu)$. If
\[
\sum_{n=1}^{\infty}\mu(E_n)<\infty,
\]
then
\[
\mu\!\left(\limsup_{n\to\infty}E_n\right)=0.
\]
Equivalently, almost every point of $\Omega$ belongs to only finitely many of the sets $E_n$.
\end{lemma}

\begin{proof}
This can be seen in \cite[Lemma~1.2]{HarmanMetric}.
\end{proof}
We shall use the following classical form of Van der Corput's second-derivative estimate.

\begin{lemma}[Second-derivative estimate]\label{lem:second}
Let $f\in C^2([M,2M])$ and suppose that
\[
\lambda\le |f''(x)|\le C\lambda\qquad (M\le x\le2M).
\]
Then
\begin{equation}\label{eq:second-derivative}
\sum_{M<n\le2M}e(f(n))\ll_C M\lambda^{1/2}+\lambda^{-1/2}.
\end{equation}
\end{lemma}

\begin{proof}
This is the classical second-derivative estimate; see, for example, \cite[Chapter~2]{GrahamKolesnik}.
\end{proof}

Next, we shall introduce the specialization of classical Van der Corput's $B$-process for the function $Cx^\gamma$ that will be used in Section~4.
\begin{lemma}[Van der Corput $B$-process]\label{lem:Bprocess}
Let $0<\gamma<1$, let $w$ be a fixed smooth function supported on a fixed compact subinterval of $(0,\infty)$, and put
\[
S(C)=\sum_n w(n/G)e(Cn^\gamma).
\]
Suppose $W:=|C|G^\gamma>G$. Then, with
\[
\alpha=\frac1{1-\gamma},
\qquad
\beta=\frac{\gamma}{1-\gamma},
\qquad
K\asymp_\gamma\frac WG,
\]
we have
\[
S(C)=\frac{G}{\sqrt W}\sum_{|k|\asymp K}a_ke\!\left(\operatorname{sgn}(C)c_\gamma|C|^\alpha |k|^{-\beta}\right)+O_\gamma\!\left(\frac{G}{\sqrt W}+1\right),
\]
where $c_\gamma\ne0$, and the coefficients $a_k$ arise from a uniformly smooth weight and satisfy $a_k\ll_\gamma1$.
\end{lemma}

\begin{proof}
This is the classical Van der Corput $B$-process applied to $f(x)=Cx^\gamma$; see \cite[Chapter~3]{GrahamKolesnik}.
The stationary point corresponding to the dual variable $k$ satisfies
\[
|k|\asymp_\gamma |C|G^{\gamma-1}\asymp_\gamma \frac{W}{G}.
\]
Substitution of this stationary point into $f(x)-kx$ gives the dual phase
\[
\operatorname{sgn}(C)c_\gamma|C|^{1/(1-\gamma)}|k|^{-\gamma/(1-\gamma)}.
\]
Moreover,
\[
|f''(x)|\asymp_\gamma WG^{-2}\qquad (x\asymp G),
\]
so the stationary-phase factor is $G/\sqrt W$. The usual $B$-process error is
\[
O_\gamma\!\left(\frac{G}{\sqrt W}+1\right).
\]
\end{proof}

The following form of the double large sieve will be used repeatedly.

\begin{lemma}[Double large sieve]\label{lem:DLS}
Let $x_i$ and $y_j$ lie in fixed bounded real intervals, $T\ge1$, and let $|a_i|,|b_j|\le1$. Put
\[
B_x(T):=\#\{(i,i'):|x_i-x_{i'}|\le T^{-1}\},
\]
and
\[
B_y(T):=\#\{(j,j'):|y_j-y_{j'}|\le T^{-1}\}.
\]
Then
\begin{equation}\label{eq:DLS}
\left|\sum_i\sum_j a_i b_j e(Tx_i y_j)\right|\ll T^{1/2}B_x(T)^{1/2}B_y(T)^{1/2}.
\end{equation}
\end{lemma}

\begin{proof}
  This is the standard double-large-sieve inequality of Bombieri and Iwaniec; see \cite{BombieriIwaniec}.
\end{proof}

Finally, we recall the spacing theorem of Robert and Sargos. For a fixed $\theta\in\mathbb{R}\setminus\{0,1\}$, $N\ge2$ and
$\delta>0$, define
\[
\mathcal{R}_\theta(N,\delta):=\#\left\{(n_1,n_2,n_3,n_4)\in(N,2N]^4:\left|n_1^\theta+n_2^\theta-n_3^\theta-n_4^\theta\right|\le\delta N^\theta\right\}.
\]

\begin{theorem}[Robert-Sargos]\label{thm:RS}
For every $\varepsilon>0$, the number of quadruples $n_1,n_2,n_3,n_4\sim N$ is bounded by
\begin{equation}\label{eq:RS}
\mathcal{R}_\theta(N,\delta)\ll_{\theta,\varepsilon}N^{2+\varepsilon}+\delta N^{4+\varepsilon}.
\end{equation}
\end{theorem}

\begin{proof}
  The proof can be traced back to \cite[Theorem~2]{RobertSargos}.
\end{proof}

\section{Better than square-root cancellation as $c\to1$}

\subsection{Proof of Theorem~\ref{thm:near}}

Throughout this section, put $M_c(X):=|\mathcal{N}^{(c)}(X)|$ and
\[
\tau:=(1-\gamma)\log X=\frac{c-1}{c}\log X
\]
for the convenience of writing. Since $n\mapsto\lfloor n^c\rfloor$ is strictly increasing, we know that
\[
M_c(X)=X^\gamma+O(1)=Xe^{-\tau}+O(1),
\]
and hence by Taylor's expansion
\begin{equation}\label{eq:deletion-size}
|D_c(X)|=X(1-e^{-\tau})+O(1)\ll X\tau+1\qquad (0\le\tau\le1).
\end{equation}

We shall use the following elementary orthogonality relation.

\begin{lemma}[Orthogonality]\label{lem:rmf-orth}
For positive integers $m,n$, we have
\[
\mathbb{E}\!\left[f(m)\overline{f(n)}\right]=
\begin{cases}
\mathbf{1}_{m=n}, & \text{if $f$ is Steinhaus},\\[3pt]
\mu^2(m)\mathbf{1}_{m=n}, & \text{if $f$ is Rademacher}.
\end{cases}
\]
Consequently, for every finite set $\mathcal{B}\subset\mathbb{N}$, it follows that
\begin{equation}\label{eq:set-L2}
\mathbb{E}\left|\sum_{n\in\mathcal B}f(n)\right|^2\le |\mathcal{B}|,
\end{equation}
with equality in the Steinhaus case.
\end{lemma}

\begin{proof}
In the Steinhaus case the first identity follows from complete multiplicativity and the independence of the variables $f(p)$. For the Rademacher model the same argument applies on the square-free support, while $f(n)=0$ when $\mu^2(n)=0$. Expanding the square gives \eqref{eq:set-L2}.
\end{proof}

Applying Lemma~\ref{lem:rmf-orth} to the set $D_c(X)$ and using Cauchy--Schwarz inequality, we obtain
\begin{equation}\label{eq:first-moment-comparison}
\left|\mathbb{E}|S_c(X;f)|-\mathbb{E}|S_X(f)|\right|\le |D_c(X)|^{1/2}.
\end{equation}

\begin{proof}[Proof of Theorem~\ref{thm:near}]
Combining Harper's first-moment estimate \cite[Theorems~1 and~2]{HarperRMF}
\[
\mathbb{E}|S_X(f)|\ll\frac{\sqrt X}{(\log\log X)^{1/4}}
\]
with \eqref{eq:first-moment-comparison}, we get the following estimate by triangle inequality
\[
\mathbb{E}|S_c(X;f)|\ll\frac{\sqrt X}{(\log\log X)^{1/4}}+|D_c(X)|^{1/2}.
\]

If $(c-1)\log X\to0$ as $X\to\infty$, then $\tau\to0$, and hence $M_c(X)\sim X$. By \eqref{eq:deletion-size}, it follows
\[
\frac{|D_c(X)|}{M_c(X)}\ll\tau+\frac{1}{X}\longrightarrow0.
\]
Dividing the preceding estimate by $\sqrt{M_c(X)}$ we can prove the theorem.
\end{proof}

\subsection{Proofs of the corollaries}

\begin{proof}[Proof of Corollary~\ref{cor:comparison}]
Under the conditions we give, it can be deduced that as $X\to\infty$,
\[
\tau\sqrt{\log\log X}\longrightarrow0.
\]
Therefore, by \eqref{eq:deletion-size}, we know that
\[
|D_c(X)|^{1/2}=o\!\left(\frac{\sqrt{X}}{(\log\log X)^{1/4}}\right).
\]
According to Harper's theorem \cite[Theorems~1 and~2]{HarperRMF}
\[
\mathbb E|S_X(f)|
\asymp
\frac{\sqrt X}{(\log\log X)^{1/4}},
\]
and \eqref{eq:first-moment-comparison}, it yields
\[
\mathbb{E}|S_c(X;f)|=(1+o(1))\mathbb{E}|S_X(f)|,
\]
which proves the corollary.
\end{proof}

For the character sum application, write
\[
\mathbb{E}_p F(\chi):=\frac{1}{p-1}\sum_{\chi\bmod p}F(\chi),
\]
and put
\[
T_X(\chi)=\sum_{n\le X}\chi(n),\qquad T_c(X;\chi)=\sum_{n\in\mathcal{N}^{(c)}(X)}\chi(n).
\]

\begin{proof}[Proof of Corollary~\ref{thm:characters}]
Since $X<p$, character orthogonality gives, for $m,n\le X$,
\[
\mathbb E_p
\bigl[\chi(m)\overline{\chi(n)}\bigr]
=
\mathbf 1_{m=n}.
\]
Consequently, by triangle and Cauchy--Schwarz inequality, we obtain
\[
\left|\mathbb{E}_p|T_c|-\mathbb{E}_p|T_X|\right|\le|\mathbb{E}_p|T_X-T_c||\le(\mathbb{E}_p|T_X-T_c|^2)^{1/2}=|D_c(X)|^{1/2}.
\]

Let $L=\min\{X,p/X\}$. Harper's estimate \cite[Theorem~1]{HarperChar} gives
\[
\mathbb{E}_p|T_X(\chi)|\ll\frac{\sqrt{X}}{\{\log\log(10L)\}^{1/4}}.
\]
Under the assumptions of the corollary, if $X\to \infty$, then we have $L\to\infty$, $M_c(X)\sim X$, and $|D_c(X)|=o(X)$ by \eqref{eq:deletion-size}. Together with Harper's estimate and the preceding comparison, this gives
\[
\mathbb{E}_p|T_c(X;\chi)|=o\!\left(\sqrt{M_c(X)}\right),
\]
which proves Corollary~\ref{thm:characters}.
\end{proof}

From Corollary~\ref{thm:characters}, we can also obtain that almost all characters satisfying $T_c(X;\chi)$ has a square-root scale upper bound. Indeed, for every $\delta>0$, Markov's inequality gives
\[
\frac{1}{p-1}\#\left\{\chi\bmod p:|T_c(X;\chi)|>\delta\sqrt{M_c(X)}\right\}\le\frac{\mathbb{E}_p|T_c(X;\chi)|}{\delta\sqrt{M_c(X)}}.
\]
The right-hand side tends to zero for every fixed $\delta>0$. Moreover, choosing $\delta=\delta_X\to0$ sufficiently slowly, we still have
\[
\frac{\mathbb{E}_p|T_c(X;\chi)|}{\delta_X\sqrt{M_c(X)}}\longrightarrow0.
\]
Hence
\[
T_c(X;\chi)=o\!\left(\sqrt{M_c(X)}\right)
\]
for all characters modulo $p$ with at most $o(p)$ exceptions.

\section{Gaussian behavior on Piatetski-Shapiro sequences}
\subsection{Reduction of the problem}\label{sec:reduction}

We begin with the standard parametrization of the solutions to
\[
m_1m_2=n_1n_2,
\]
which is also recorded in \cite[Lemma~3.4]{SoundXu}.

\begin{lemma}[Multiplicative rectangles]\label{lem:rectangles}
Every positive-integer solution of $m_1m_2=n_1n_2$ has a unique representation
\begin{equation}\label{eq:rectangles}
m_1=ga,\qquad m_2=hb,\qquad n_1=gb,\qquad n_2=ha,\qquad (a,b)=1.
\end{equation}
The solution is outside the two diagonal families if and only if $a\ne b$ and $g\ne h$. Thus, for every finite $\mathcal{A}\subset\mathbb{N}$, we have
\begin{equation}\label{eq:offidentity}
E^\times(\mathcal{A})=2|\mathcal{A}|^2-|\mathcal{A}|+N_{\mathrm{off}}(\mathcal{A}),
\end{equation}
where $N_{\mathrm{off}}(\mathcal{A})$ counts the tuples in \eqref{eq:rectangles} satisfying $a\ne b$, $g\ne h$.
\end{lemma}

\begin{proof}
Take
\[
g=(m_1,n_1),\qquad a=\frac{m_1}{g},\qquad b=\frac{n_1}{g}.
\]
Then $(a,b)=1$ together with the equation $m_1m_2=n_1n_2$ gives
\[
am_2=bn_2.
\]
Hence $m_2=hb$ and $n_2=ha$ for a unique positive integer $h$. It is worth noting that the two diagonal families correspond respectively to $a=b=1$ or $g=h$ while the intersection of these two cases contains $|\mathcal{A}|$ elements. This contributes to the main term in \eqref{eq:offidentity} and the error term can be named as $N_{\mathrm{off}}(\mathcal{A})$. 
\end{proof}

Next we fix a compact interval $I=[\gamma_0,\gamma_1]\subset(7/8,1)$. The variables $a,b,g,h$ can be decomposed into dyadic intervals. Thus all scale parameters below are powers of $2$, which means $Y=2^j$ for some integer $j$. By interchanging the variables within each pair if necessary, we may write
\[
a\asymp A,\qquad b\asymp \rho A,\qquad g\asymp G,\qquad h\asymp \sigma G,
\]
where
\[
0<\rho,\sigma\le1.
\]

Put $Y=AG$ and then it follows that
\[
ga\asymp Y,\qquad gb\asymp \rho Y,\qquad ha\asymp \sigma Y,\qquad hb\asymp \rho\sigma Y.
\]
Therefore, for further discussion we first need to control the contribution from the unbalanced blocks for which one of the ratios $\rho,\sigma$ is very small. Choose a sufficiently large constant $B_I$ and define
\begin{equation}\label{eq:Rdef}
\mathcal{R}(Y)=\exp\!\left(\frac{B_I\log Y}{\log\log Y}\right).
\end{equation}

\begin{lemma}\label{lem:extreme}
The total contribution of the dyadic blocks satisfying
\[
\rho<\mathcal{R}(Y)^{-1}\qquad\text{or}\qquad\sigma<\mathcal{R}(Y)^{-1}
\]
is $o(X^{2\gamma})$.
\end{lemma}

\begin{proof}
Suppose first that $\rho<\mathcal{R}(Y)^{-1}$. The two cross entries $gb$ and $ha$ have respective sizes $\asymp\rho Y$ and $\asymp\sigma Y$. Since all four entries belong to $\mathcal{N}^{(c)}(X)$, the number of possible choices for these two entries is
\[
\ll_I(\rho Y)^\gamma(\sigma Y)^\gamma.
\]
Once $gb$ and $ha$ are fixed, the remaining two entries $ga$ and $hb$ form an ordered factorization of
\[
(gb)(ha)=(ga)(hb).
\]

By the classical divisor bound, the number of such factorizations is
\[
\ll\exp\!\left(C\frac{\log Y}{\log\log Y}\right)
\]
for some absolute constant $C>0$. Hence the contribution of the block is
\[
\ll_I(\rho\sigma)^\gamma Y^{2\gamma}\exp\!\left(C\frac{\log Y}{\log\log Y}\right).
\]
Since $\sigma\le1$ and $\rho<\mathcal{R}(Y)^{-1}$, this is
\[
\ll_I Y^{2\gamma}\exp\!\left(-\frac{(B_I\gamma-C)\log Y}{\log\log Y}\right).
\]
The same estimate holds when $\sigma<\mathcal R(Y)^{-1}$.

Choose $B_I$ so that $B_I\gamma_0>C+10$. The number of dyadic choices at each scale is bounded by a power of $\log Y$, and summing over the dyadic values $Y\ll X$ therefore gives $o(X^{2\gamma})$.
\end{proof}

Throughout the following subsections, we work at a fixed dyadic scale $Y$. Our aim is to show that, for almost every fixed $\gamma\in I$, the total off-diagonal contribution at this scale is
\[
\omega_\gamma(Y)Y^{2\gamma},\qquad \omega_\gamma(Y)\longrightarrow0.
\]
Since $Y$ ranges over dyadic values with $Y\ll X$, summing these estimates over $Y$ gives $o(X^{2\gamma})$. We shall carry out this final dyadic summation at the end of the proof. Therefore, it only remains to consider the balanced blocks, by which we mean
\[
\mathcal{R}(Y)^{-1}\le \rho,\sigma\le 1.
\]
Thus the two variables in each pair have dyadic scales differing by at most a factor $\mathcal{R}(Y)$. By transposing the rectangle if necessary, and discarding the coprimality condition when taking an upper bound, we may further assume $A\le G$.
Since $Y=AG$, it follows that $A\le \sqrt Y\le G$. Finally, every fixed power of $\mathcal{R}(Y)$ is $Y^{o(1)}$, and such factors will be absorbed into $Y^{o(1)}$ below.

\subsection{Fourier expansion on the balanced blocks}\label{sec:fourierrows}

We keep the notation of Section~\ref{sec:reduction}. For the convenience of writing, we put
\[
A_1=A,\qquad A_2=\rho A,\qquad
G_1=G,\qquad G_2=\sigma G.
\]
Thus it follows that
\[
a\sim A_1,\qquad b\sim A_2,\qquad
g\sim G_1,\qquad h\sim G_2,
\]
where $\mathcal{R}(Y)^{-1}\le \rho,\sigma\le1$, $A\le G$ and $AG=Y$.

For the four Piatetski-Shapiro indicators in a fixed balanced block, set
\[
Z_{11}=A_1G_1,\qquad Z_{21}=A_2G_1,\qquad
Z_{12}=A_1G_2,\qquad Z_{22}=A_2G_2,
\]
and let $\Psi_{ij}$ be the corresponding majorants given by Lemma~\ref{lem:majorant}. Thus, we obtain
\[
\begin{aligned}
\chi_\gamma(ag)\chi_\gamma(bg)\chi_\gamma(ah)\chi_\gamma(bh)\le\Psi_{11}(-a^\gamma g^\gamma)\Psi_{21}(-b^\gamma g^\gamma)\Psi_{12}(-a^\gamma h^\gamma)\Psi_{22}(-b^\gamma h^\gamma).
\end{aligned}
\]
Put $H=Y^{1-\gamma}$ and $H_{ij}=Z_{ij}^{1-\gamma}$. Since the block is balanced, it naturally has
\[
H\mathcal{R}(Y)^{-2(1-\gamma_0)}\ll H_{ij}\le H.
\]

Fix a sufficiently small $\varepsilon>0$. By Lemma~\ref{lem:majorant}, for any fixed $A_*>0$ we may write
\[
\Psi_{ij}(-t)=\sum_{|r|\le H_{ij}Y^\varepsilon}\widehat{\Psi}_{ij}(-r)e(rt)+R_{ij}(t),
\]
where 
\[
\|R_{ij}\|_\infty\ll_{I,\varepsilon,A_*}Y^{-A_*},\qquad\sum_{r\in\mathbb{Z}}|\widehat{\Psi}_{ij}(r)|\ll_I1.
\]

Write $\mathbf{r}=(r_1,s_1,r_2,s_2)$. Multiplying the four truncated expansions gives the following formula
\[
\begin{aligned}
&\Psi_{11}(-a^\gamma g^\gamma)\Psi_{21}(-b^\gamma g^\gamma)\Psi_{12}(-a^\gamma h^\gamma)\Psi_{22}(-b^\gamma h^\gamma)\\
&\quad=\sum_{\mathbf{r}\in\mathcal{F}}c(\mathbf{r})\,e\!\left(g^\gamma(r_1a^\gamma+s_1b^\gamma)+h^\gamma(r_2a^\gamma+s_2b^\gamma)\right)+O_{I,\varepsilon,A_*}(Y^{-A_*}).
\end{aligned}
\]
In this equality, the finite set $\mathcal{F}$ is determined by
\[
\begin{aligned}
|r_1|&\le H_{11}Y^\varepsilon,
\quad |s_1|&\le H_{21}Y^\varepsilon,
\quad|r_2|&\le H_{12}Y^\varepsilon,
\quad |s_2|&\le H_{22}Y^\varepsilon.
\end{aligned}
\]
We also write the Fourier coefficients for short as
\[
c(\mathbf{r})=\widehat{\Psi}_{11}(-r_1)\widehat{\Psi}_{21}(-s_1)\widehat{\Psi}_{12}(-r_2)\widehat{\Psi}_{22}(-s_2).
\]
In particular, according to Lemma~\ref{lem:majorant}, the coefficient vector satisfies that
\[
\sum_{\mathbf{r}\in\mathcal{F}}|c(\mathbf{r})|
\ll_I1,
\]
and
\begin{equation}\label{eq:cutoffmetric}
|r_j|,|s_j|
\ll
Y^{1-\gamma_0+\varepsilon}
\qquad (j=1,2).
\end{equation}

The number of parameter tuples in the block is at most
\[
A_1A_2G_1G_2\ll Y^2.
\]
Thus, the total contribution of the error term in Fourier expansion is bounded as $O(Y^{-A_*+2})$. On taking $A_*$ sufficiently large, we may control the estimate as $O(Y^{-10})$ and then focus only on the finite exponential sum over $\mathbf{r}\in\mathcal{F}$.
 
Next we put $C_1=r_1a^\gamma+s_1b^\gamma,\ C_2=r_2a^\gamma+s_2b^\gamma$ for a fixed frequency tuple. Choose fixed non-negative smooth functions $w_1,w_2$ supported on fixed enlargements of the dyadic intervals for $g,h$, and write the exponential sum as
\[
S_1(C_1)=\sum_g w_1(g/G_1)e(C_1g^\gamma),
\qquad
S_2(C_2)=\sum_h w_2(h/G_2)e(C_2h^\gamma).
\]

Now we deal with the conditions $a\ne b$ and $g\ne h$ differently. The condition $a\ne b$ can be retained throughout the following part. Then noticing that
\[
\begin{aligned}
&\sum_{g\ne h}w_1(g/G_1)w_2(h/G_2)e(C_1g^\gamma+C_2h^\gamma)\\
&\qquad=
S_1(C_1)S_2(C_2)-\sum_nw_1(n/G_1)w_2(n/G_2)e((C_1+C_2)n^\gamma),
\end{aligned}
\]
summation over $g\ne h$ can be asymptotic to summation over $g$ times summation over $h$ because the diagonal condition $g=h$ can bound the last sum as $O(G)$. Using the fact $A\le\sqrt Y$ and $AG=Y$, its total cost is less than $A^2G\le Y^{3/2}=o(Y^{2\gamma})$.

Consequently, after disposing of the unbalanced blocks, the Fourier remainders and the diagonal condition $g=h$, we finally reduce to the problem to such an exponential sum
\[
\mathcal{B}:=\sum_{\mathbf{r}\in\mathcal{F}}|c(\mathbf{r})|\sum_{\substack{a\sim A_1,\ b\sim A_2\\ a\ne b}}|S_1(C_1)S_2(C_2)|.
\]
Then we begin with the classical Van der Corput's second-derivative estimate..

\subsection{The second-derivative range}\label{sec:medium}
We now consider the exponential sums obtained in Section~\ref{sec:fourierrows}. For $j=1,2$, recall that
\[
S_j(C_j)=\sum_n w_j(n/G_j)e(C_jn^\gamma).
\]
For $C_j\ne0$, put 
\begin{equation}\label{eq:Wdef}
W_j:=|C_j|G_j^\gamma.
\end{equation}
If we define
\[
f_j(x)=C_jx^\gamma,
\]
then, for $x\asymp G_j$, the second derivative of $f$ satisfies that
\[
|f_j''(x)|
\asymp_\gamma |C_j|G_j^{\gamma-2}
=W_jG_j^{-2}.
\]
Thus $W_j$ is exactly the parameter that measures the size of the second derivative on the scale $G_j$. We first apply the classical Van der Corput's second-derivative estimate to exponential sum $S_j(C_j)$. By Lemma~\ref{lem:second}, it follows
\begin{equation}\label{eq:secondW}
S_j(C_j)\ll_\gamma G_jW_j^{-1/2}+W_j^{1/2}.
\end{equation}

This estimate is effective when $W_j$ is neither too small nor too large. Let $\eta>0$ be a parameter chosen later, subject to
\[
\eta>2(1-\gamma_0).
\]
Suppose first that
\[
Y^\eta<W_j\le G_j.
\]
Since $W_j\le G_j$, the second term in \eqref{eq:secondW} is dominated by the first, and hence
\begin{equation}\label{eq:mediumsingle}
S_j(C_j)
\ll_\gamma G_jW_j^{-1/2}
\ll G_jY^{-\eta/2}.
\end{equation}
Then we can obtain that
\[
|S_1(C_1)S_2(C_2)|\ll_\gamma G_1G_2Y^{-\eta}.
\]
Hence, combining the definition of $A_j, G_j$ for $j=1,2$ with $AG=Y$, the estimate of the exponential sum for each fixed $\mathbf{r}\in\mathcal{F}$ satisfies that
\[
\sum_{\substack{a\sim A_1,\ b\sim A_2\\a\ne b,\; Y^\eta<W_j\le G_j\ (j=1,2)}}|S_1(C_1)S_2(C_2)|\ll_\gamma A_1A_2G_1G_2Y^{-\eta}\ll Y^{2-\eta}.
\]
Summing next over the Fourier frequencies and using $\sum_{\mathbf{r}\in\mathcal{F}}|c(\mathbf{r})|\ll_I1$, we obtain
\[
\sum_{\mathbf{r}\in\mathcal{F}}|c(\mathbf{r})|\sum_{\substack{a\sim A_1,\ b\sim A_2\\a\ne b,\; Y^\eta<W_j\le G_j\ (j=1,2)}}|S_1(C_1)S_2(C_2)|\ll_{\gamma,I}Y^{2-\eta}=o(Y^{2\gamma}),
\label{eq:mediumfinal}
\]
since $\eta>2(1-\gamma_0)$ and $\gamma\ge\gamma_0$.

Thus the classical second-derivative estimate is sufficient when both $W_1$ and $W_2$ lie in the intermediate range. However, if only one of them lies in this range while the other exponential sum is estimated trivially, then we obtain only $Y^{2-\eta/2}$ as the upper bound for the product of two exponential sums. It means one would need $\eta>4(1-\gamma)$ so that the contribution is $o(Y^{2\gamma})$, which is stronger than the condition available to us and would lead to a smaller admissible range of $\gamma$. Therefore, we have to treat cases in which at least one $W_j$ leaves the intermediate range by different arguments.

If $0<W_j\le Y^\eta$ for some $j\in\{1,2\}$, then $C_j$ may be very small. So the first term in \eqref{eq:secondW} no longer gives a uniform saving. Luckily, since $C_j$ is a sum of two terms coming from the Fourier frequencies, such a small value can occur only through a near cancellation between them. This leads to a Diophantine problem in the variables $a,b$ and the frequencies, which will be treated in the next subsection.

On the other hand, if $W_j>G_j$ for some $j\in\{1,2\}$, then the second term in \eqref{eq:secondW} can no longer be absorbed
by the first. In this range we shall instead use the Van der Corput's $B$-process and double large sieve method.

Accordingly, the advantages and drawbacks of classical Van der Corput's second-derivative method naturally splits the summation into three parts
\[
0<W_j\le Y^\eta,\qquad Y^\eta<W_j\le G_j,\qquad W_j>G_j.
\]
The middle range is settled above. We next consider the small-$W$ range, where the main issue is the potential near cancellation in $C_j$.

\subsection{The small-$W$ range}\label{sec:smallW}
We now consider the small range $0<W_j\le Y^\eta$ for some $j\in\{1,2\}$.

First, we need to complete the choice of the parameter $\eta$ introduced in the preceding subsection. Put $s_0:=1-\gamma_0$. Since $\gamma_0>7/8$, we may choose a sufficiently small $\kappa>0$ such that $8s_0+2\kappa<1$, and then fix $\eta$ so that
\begin{equation}\label{eq:smallWparameters}
2s_0<\eta<1-6s_0-2\kappa.
\end{equation}
In particular, $\eta<\gamma_0$. We also put $A_0:=Y^{2s_0+\kappa}$.

From the simple observation, it can be deduced that small-$W_j$ condition can occur only when $C_j$ is small, which means the Fourier frequencies have cancellation. Indeed, recall that
\[
C_j=r_ja^\gamma+s_jb^\gamma.
\]
If one of $r_j,s_j$ is zero, or if $r_j$ and $s_j$ have the same sign, then there is no cancellation between the two terms. By the preceding bounds for the Fourier frequencies and the balanced scales, this gives
\[
W_j\gg Y^{\gamma-o(1)}>Y^\eta,
\]
which contradicts the assumption above. Hence both Fourier frequencies must be nonzero and have opposite signs, which means we can write $C_j$ as
\[
C_j=r_ja^\gamma-s_jb^\gamma,\qquad r_j,s_j\ge1.
\]
Since $W_j=G_j^\gamma|C|\le Y^\eta$, the two terms above must be close. In particular, we consider
\begin{equation}\label{eq:cancellationratio}
\frac{r_j}{s_j}=\left(\frac{b}{a}\right)^\gamma(1+o(1)),\qquad\min(r_j,s_j)\gg\max(r_j,s_j)\mathcal{R}(Y)^{-O_I(1)}.
\end{equation}

There are two useful ways dependent on the size of $A$ to exploit this near cancellation. When $A$ is small, we estimate the measure of the set of exponents $\gamma$ for which such a cancellation can occur and then use the first Borel--Cantelli lemma.  The corresponding union bound has a cost essentially of size $A^2$, and is therefore effective only on the shorter range. When $A$ is large, it is more efficient to fix $a,r_j,s_j$ and count the possible integers $b$ directly from the monotonicity of $r_ja^\gamma-s_jb^\gamma$. The resulting bound contains a term $Y^2/A$, so the second argument becomes stronger as $A$ grows.
The parameter $A_0$ is chosen so that these two arguments overlap with a power saving.

Before carrying this out, we need to control the contribution from exact cancellations $r_ja^\gamma=s_jb^\gamma$. Define
\begin{equation}\label{eq:exactexception}
\mathcal{E}_{\mathrm{exact}}:=\left\{\gamma\in(0,1):ra^\gamma=sb^\gamma\text{ for some }a\ne b,\ r,s\in\mathbb{N}\right\}.
\end{equation}
This set is countable. Indeed, for each fixed quadruple $(a,b,r_j,s_j)$ with $a\ne b$, the equation has at most one solution, namely
\[
\gamma=\frac{\log(s/r)}{\log(a/b)}.
\]
Hence $\mathcal{E}_{\mathrm{exact}}$ has Lebesgue measure zero. In what follows we work outside this set, so a nonzero frequency pair always gives $W_j>0$.

\subsubsection{Small \texorpdfstring{$A$}{A}: metric exclusion}

We first record the elementary measure estimate that will be used in the metric argument.

\begin{lemma}\label{lem:transversality}
Let $1\le a<b$, $r,s>0$, and $0<\delta<1/10$. Then
\begin{equation}\label{eq:measurelemma}
\meas\left\{\gamma\in I:|ra^\gamma-sb^\gamma|\le\delta(ra^\gamma+sb^\gamma)\right\}\ll\frac{\delta b}{b-a}.
\end{equation}
\end{lemma}

\begin{proof}
Put
\[
Q(\gamma):=\frac{r}{s}\left(\frac{a}{b}\right)^\gamma.
\]
The condition in \eqref{eq:measurelemma} implies
\[
\frac{1-\delta}{1+\delta}\le Q(\gamma)\le\frac{1+\delta}{1-\delta}.
\]

Hence after taking the logarithm of both sides of the inequality, we know that
\[
\log(r/s)+\gamma\log(a/b)
\]
lies in an interval of length
\[
2\log\frac{1+\delta}{1-\delta}\ll\delta.
\]

From the fact that the logarithm is a linear function of $\gamma$ with the absolute value of its slope being $\log(b/a)$, we obtain
\[
\meas\left\{\gamma\in I:|ra^\gamma-sb^\gamma|\le\delta(ra^\gamma+sb^\gamma)\right\}\ll\frac{\delta}{\log(b/a)}.
\]
Finally, using the basic skill
\[
\log(b/a)=\int_a^b\frac{\mathrm dt}{t}\ge\frac{b-a}{b}.
\]
and the result follows.
\end{proof}

We now determine how small the size of $\delta$ in \ref{lem:transversality} has to be before a metric argument is needed. Considering the phase $C=ra^\gamma-sb^\gamma$ for $r,s>0$, put
\begin{equation}\label{eq:relativegap}
\Delta(\gamma):=\frac{|ra^\gamma-sb^\gamma|}{ra^\gamma+sb^\gamma}.
\end{equation}

Fix a dyadic frequency range $Q\le \max(r,s)<2Q$. By \eqref{eq:cancellationratio}, $r$ and $s$ are comparable up to a
fixed power of $\mathcal R(Y)$, while on a balanced block $a$ and $b$ are comparable to $A$ up to the same type of factor. Hence we have the lower bound for the denominator in \eqref{eq:relativegap}
\[
ra^\gamma+sb^\gamma\gg_IQA^\gamma\mathcal{R}(Y)^{-O_I(1)}.
\]
Moreover,
\[
G_j\ge\frac{Y}{A\mathcal{R}(Y)^{O_I(1)}}.
\]
Therefore, it follows that
\begin{equation}\label{eq:WfromDelta}
W_j=G_j^\gamma|ra^\gamma-sb^\gamma|\gg_I\Delta(\gamma)\,QY^\gamma\mathcal{R}(Y)^{-O_I(1)}.
\end{equation}

We can now introduce the threshold $\delta_Q:=C_IQ^{-1}Y^{\eta-\gamma_0}\mathcal{R}(Y)^{O_I(1)}$, where $C_I>0$ is chosen sufficiently large. Since $\gamma\ge\gamma_0$, \eqref{eq:WfromDelta} shows that if $\Delta(\gamma)>\delta_Q$, then it can be deduced that $W_j>Y^\eta$.

Thus the small $W$-range $0<W_j\le Y^\eta$ we required in this section can occur only if $\Delta(\gamma)\le\delta_Q$. Indeed, if $\Delta(\gamma)>\delta_Q$, then $W_j>Y^\eta$. If in addition $W_j\le G_j$, the second-derivative estimate of Section~\ref{sec:medium} applies; the remaining case $W_j>G_j$ is treated in Section~\ref{sec:largeW}. Hence, in the small-$A$ range, the only new obstruction is the occurrence of $\Delta(\gamma)\le\delta_Q$.

For a dyadic $Y$, let $E_Y$ be the set of parameters $\gamma\in I$ for which there exist a balanced block with $A\le A_0$, integers $a\ne b$, and an admissible opposite-sign nonzero frequency pair in some dyadic shell $Q$ such that
\[
\Delta(\gamma)\le\delta_Q.
\]
Next we shall show that these exceptional sets have summable measure.

Since $\eta<\gamma_0$ and every fixed power of $\mathcal{R}(Y)$ is $Y^{o(1)}$, we have $\delta_Q<\frac{1}{10}$ for all sufficiently large $Y$. For fixed $a,b,r,s$, interchanging $a,b$ if necessary, we may assume $a<b$. Lemma~\ref{lem:transversality} then gives
\[
\meas\left\{\gamma\in I:\Delta(\gamma)\le\delta_Q\right\}\ll_I\delta_Q\frac{b}{b-a}.
\]
Since $b\ll A\mathcal{R}(Y)^{O_I(1)}$, we obtain
\begin{equation}\label{eq:fixedmeasure}
\meas\left\{\gamma\in I:\Delta(\gamma)\le\delta_Q\right\}\ll_I\delta_Q\frac{A\mathcal{R}(Y)^{O_I(1)}}{|a-b|}.
\end{equation}

We now sum \eqref{eq:fixedmeasure} over all possible choices. For fixed dyadic $A$ and $Q$,
\[
\sum_{\substack{a,b\ll A\mathcal{R}(Y)^{O_I(1)}\\a\ne b}}\frac{A\mathcal{R}(Y)^{O_I(1)}}{|a-b|}\ll_IA^2\mathcal{R}(Y)^{O_I(1)}\log(2Y).
\]
There are $O(Q^2)$ possible frequency pairs $(r,s)$ in the dyadic shell, and the factor $Q^{-1}$ comes from $\delta_Q$. Hence, for fixed $A$ and $Q$, the total measure is $\ll_I A^2 Q Y^{\eta-\gamma_0+o(1)}$.

By \eqref{eq:cutoffmetric}, we know that
\[
Q\ll Y^{s_0+\epsilon},\qquad s_0=1-\gamma_0.
\]
Since both $A$ and $Q$ run over dyadic values, it has
\[
\sum_Q Q\ll Y^{s_0+\epsilon},\qquad\sum_{A\le A_0}A^2\ll A_0^2.
\]
The remaining dyadic choices contribute only powers of $\log Y$. Recalling that $A_0=Y^{2s_0+\kappa}$, we conclude that
\begin{align}
\meas(E_Y)&\ll_IA_0^2Y^{\eta-\gamma_0+s_0+\epsilon+o(1)}\nonumber\\&=Y^{\eta-1+6s_0+2\kappa+\epsilon+o(1)}.
\label{eq:metricbound}
\end{align}

According to the range of $\eta<1-6s_0-2\kappa$ given in \eqref{eq:smallWparameters}, we may choose $\epsilon>0$ sufficiently small so that
\[
\eta-1+6s_0+2\kappa+2\epsilon<0.
\]
After absorbing the $o(1)$ term, there exists $\delta_0>0$ such that
\[
\meas(E_Y)\ll_I Y^{-\delta_0}
\]
for all sufficiently large $Y$. Since $Y=2^j$ is dyadic, it follows
\[
\sum_{j=1}^{\infty}\meas(E_{2^j})\ll\sum_{j=1}^{\infty}\frac{1}{2^{j\delta_0}}<\infty.
\]
The first Borel--Cantelli lemma therefore gives
\[
\meas\left(\limsup_{j\to\infty}E_{2^j}\right)=0.
\]

Thus, for almost every $\gamma\in I$, sufficiently large dyadic scales contain no small-$A$ near cancellation with $\Delta(\gamma)\le\delta_Q$. In particular, no nonzero small-$W$ phase with $A\le A_0$ occurs eventually.

\subsubsection{Large \texorpdfstring{$A$}{A}: deterministic counting}

In this subsection it remains to treat $A>A_0$. Fix $a,r,s$ with $Q\le \max(r,s)<2Q$. By the balanced block conditions, we have
\[
\left|\frac{\mathrm{d}}{\mathrm{db}}(ra^\gamma-sb^\gamma)\right|=\gamma s b^{\gamma-1}\gg_I QA^{\gamma-1}\mathcal{R}(Y)^{-O_I(1)}.
\]
Since $W_j\le Y^\eta$ implies
\[
|ra^\gamma-sb^\gamma|\ll_I A^\gamma Y^{\eta-\gamma}\mathcal{R}(Y)^{O_I(1)},
\]
monotonicity gives
\[
\#\{b\}\ll_I1+\frac{AY^{\eta-\gamma}\mathcal{R}(Y)^{O_I(1)}}{Q}.
\]

Let $H=Y^{1-\gamma}$, and let $\mathcal{C}_Q$ denote the product of the two Fourier coefficients belonging to the small phase. Then it follows that
\[
\mathcal{C}_Q\ll_IH^{-2}\mathcal{R}(Y)^{O_I(1)}\left(1+\frac{Q}{H}\right)^{-B}.
\]
After summing over $a$, the $O(Q^2)$ frequency pairs, and dyadic $Q$, the corresponding total weight which we name as $\mathcal{M}_A$ satisfies
\[
\mathcal{M}_A\ll Y^{o(1)}\left(A+A^2H^{-1}Y^{\eta-\gamma}\right).
\]

Let $\mathcal{T}_A$ denote the total contribution from the present dyadic block with fixed $A$. Summing the Fourier coefficients over the dyadic frequency ranges, and estimating the two long sums trivially, gives
\[
\begin{aligned}
\mathcal{T}_A&\ll_IG^2Y^{o(1)}\left(A+A^2H^{-1}Y^{\eta-\gamma}\right) \\
&=Y^{o(1)}\left(\frac{Y^2}{A}+Y^{1+\eta}\right) \\
&\ll Y^{2\gamma_0-\kappa+o(1)}+Y^{1+\eta+o(1)}=o(Y^{2\gamma}).
\end{aligned}
\]
Here we used $G=Y/A$, $H=Y^{1-\gamma}$, $A>A_0=Y^{2s_0+\kappa}$, and $\eta<2\gamma_0-1$.

Combining this with the small-$A$ metric argument, we conclude that, outside a null subset of $I$, all cases containing a nonzero small-$W$ phase contribute $o(Y^{2\gamma})$.

\subsection{The large-$W$ range}\label{sec:largeW}

We now treat the configurations for which $W_j>G_j$ for at least one $j\in\{1,2\}$. In this range the second-derivative estimate is no longer sufficient, and we instead apply the Van der Corput's $B$-process from Lemma~\ref{lem:Bprocess}.

We first work at comparable scales
\[
a,b\asymp A,\qquad g,h\asymp G,\qquad AG=Y.
\]
On a balanced block the actual scales differ from these by at most a fixed power of $\mathcal{R}(Y)$. Since
\[
\mathcal{R}(Y)^{O_I(1)}=Y^{o(1)},
\]
these changes will be absorbed into the factors $Y^{o(1)}$ below.

\subsubsection{The transformed phase and spacing estimates}

Fix a large phase
\[
C(a,b)=ra^\gamma+sb^\gamma,\qquad rs\ne0.
\]
We split according to the sign and dyadic size of $C(a,b)$. Thus, fix
\[
\sigma=\operatorname{sgn}(C)
\]
and a dyadic value $C_0$ such that $C_0\le |C(a,b)|<2C_0$. Put
\[
W=C_0G^\gamma,\qquad K\asymp_\gamma \frac{W}{G},
\]
and define
\[
u(a,b)=\left(\frac{|C(a,b)|}{C_0}\right)^\alpha,\qquad v(k)=\left(\frac K{|k|}\right)^\beta,\qquad T=c_\gamma C_0^\alpha K^{-\beta}\asymp_\gamma W.
\]
By Lemma~\ref{lem:Bprocess}, apart from the factor $G/\sqrt W$, a uniformly smooth weight, and the standard error term, the dual oscillation has the product form
\[
e\!\left(\sigma T u(a,b)v(k)\right).
\]

The joint stationary amplitude is a smooth function of the normalized variables. Expanding it in a Fourier series on a fixed compact box, we may write it as a sum of separated amplitudes with bounded total coefficient mass. Hence it is enough to work with arbitrary bounded coefficients in $(a,b)$ and $k$.

We shall use the following consequence of the Robert-Sargos spacing theorem.

\begin{lemma}\label{lem:fourthintegral}
Let $\theta\notin\{0,1\}$ and
\[
F_N(t)=\sum_{n\sim N}e\!\left(t(n/N)^\theta\right).
\]
For $U>0$ and $\epsilon>0$, we have
\[
\int_{-U}^{U}|F_N(t)|^4\,dt\ll_{\theta,\epsilon}N^{2+\epsilon}U+N^{4+\epsilon}.
\]
\end{lemma}

\begin{proof}
For $U<1$ the trivial estimate is sufficient. Thus we focus on the case $U\ge1$, using the fact
\[
\Phi(t)=\left(\frac{\sin\pi t}{\pi t}\right)^2.
\]
After majorizing the sharp interval by a dilate of $\Phi$, expanding the fourth power and applying Fourier inversion, the integral is bounded by $O(U)$ times the number of quadruples satisfying
\[
\left|n_1^\theta+n_2^\theta-n_3^\theta-n_4^\theta\right|\ll N^\theta U^{-1}.
\]
Then Theorem~\ref{thm:RS} gives the result.
\end{proof}

For a nonzero frequency pair put
\[
Q:=\max(|r|,|s|),\qquad\lambda:=\frac{\min(|r|,|s|)}{Q}\in(0,1].
\]

\begin{lemma}\label{lem:mixedspacing}
On the fixed-sign block above, define
\[
B_A(Z):=\#\left\{(a,b),(a',b'):|u(a,b)-u(a',b')|\le Z^{-1}\right\}.
\]
Then we have the following estimate
\[
B_A(Z)\ll_{\gamma,\epsilon}\lambda^{-1/2}\left(A^{2+\epsilon}+\frac{A^{4+\epsilon}}Z\right).
\]
\end{lemma}

\begin{proof}
Since $|C(a,b)|\asymp C_0$ on the block, a counted pair satisfies
\[
|r(a^\gamma-a'^\gamma)+s(b^\gamma-b'^\gamma)|\ll_\gamma C_0 Z^{-1}.
\]
Noticing that $C_0\ll QA^\gamma$, we can divide both sides of the inequality above by $QA^\gamma$. Therefore, it is enough to count quadruples with
\[
\left|\lambda_1\left((a/A)^\gamma-(a'/A)^\gamma\right)+\lambda_2\left((b/A)^\gamma-(b'/A)^\gamma\right)\right|\ll Z^{-1},
\]
where $|\lambda_1|=1$ and $|\lambda_2|=\lambda$.

Using the same kernel as in Lemma~\ref{lem:fourthintegral}, we obtain
\[
B_A(Z)\ll\frac{1}{Z}\int_{|t|\le C_\gamma Z}|F_A(t)|^2 |F_A(\lambda t)|^2\,dt.
\]
By Cauchy--Schwarz inequality, it follows that
\[
B_A(Z)\ll\frac{1}{Z}\left(\int_{|t|\le C_\gamma Z}|F_A(t)|^4\,dt\right)^{1/2}\left(\int_{|t|\le C_\gamma Z}|F_A(\lambda t)|^4\,dt\right)^{1/2}.
\]
Then Lemma~\ref{lem:fourthintegral} gives
\[
\int_{|t|\le C_\gamma Z}|F_A(t)|^4\,dt\ll_{\gamma,\epsilon}A^{2+\epsilon}Z+A^{4+\epsilon}.
\]
For the second integral, the change of variables $u=\lambda t$ gives
\[
\begin{aligned}
\int_{|t|\le C_\gamma Z}|F_A(\lambda t)|^4\,dt&=\lambda^{-1}\int_{|u|\le C_\gamma\lambda Z}|F_A(u)|^4\,du  \\
&\ll_{\gamma,\epsilon}A^{2+\epsilon}Z+\lambda^{-1}A^{4+\epsilon}.
\end{aligned}
\]
Therefore, we can deduce that
\[
B_A(Z)\ll_{\gamma,\epsilon}\frac{1}{Z}\bigl(A^{2+\epsilon}Z+A^{4+\epsilon}\bigr)^{1/2}\bigl(A^{2+\epsilon}Z+\lambda^{-1}A^{4+\epsilon}\bigr)^{1/2}.
\]
Since $0<\lambda\le1$, the right side satisfies
\[
\ll_{\gamma,\epsilon}\lambda^{-1/2}\left(A^{2+\epsilon}+\frac{A^{4+\epsilon}}Z\right),
\]
which proves the lemma.
\end{proof}

On the dual side, write
\[
v(k)=\left(\frac{K}{|k|}\right)^\beta,\qquad |k|\asymp K.
\]
We shall use the two spacing quantities
\[
B_K^{(2)}(T):=\#\left\{k_1,k_2\asymp K:|v(k_1)-v(k_2)|\le T^{-1}\right\},
\]
and
\[
B_K^{(4)}(T):=\#\left\{k_1,k_2,k_3,k_4\asymp K:|v(k_1)-v(k_2)-v(k_3)+v(k_4)|\le T^{-1}\right\}.
\]
Since
\[
|v'(x)|\asymp_\gamma K^{-1}\qquad (x\asymp K),
\]
the monotonicity gives
\[
B_K^{(2)}(T)\ll_\gamma K+\frac{K^2}{T}.
\]
Moreover, applying Theorem~\ref{thm:RS} with exponent $-\beta$ gives
\[
B_K^{(4)}(T)\ll_{\gamma,\epsilon}K^{2+\epsilon}+\frac{K^{4+\epsilon}}{T}.
\]

\begin{proposition}\label{prop:transfer}
For arbitrary bounded coefficients $c_{a,b}$ and $d_k$,
\[
\begin{aligned}
\left|
\sum_{a,b\asymp A}
c_{a,b}
\sum_{|k|\asymp K}
d_k e(Tu(a,b)v(k))
\right|
\ll_{\gamma,\epsilon}
&\lambda^{-1/4}(A^2K)^{1+\epsilon} \\
&\times
\left\{
\left(\frac{T}{A^2K^2}\right)^{1/4}
+A^{-1/2}
+K^{-1/2}
+T^{-1/2}
\right\}.
\end{aligned}
\]
The same estimate holds in the presence of the smooth joint amplitude
coming from the $B$-process.
\end{proposition}

\begin{proof}
Denote the sum in the proposition by
\[
S:=\sum_{a,b\asymp A}c_{a,b}\sum_{k\asymp K}d_k e\!\left(Tu(a,b)v(k)\right).
\]

We first consider the range $1\le T\le A^2$. Apply the double large sieve directly to the two families $x_{a,b}=u(a,b)$ and $y_k=v(k)$. Lemma \ref{lem:DLS} gives
\[
|S|\ll T^{1/2}B_A(T)^{1/2}B_K^{(2)}(T)^{1/2}.
\]
Since $T\le A^2$, Lemma~\ref{lem:mixedspacing} gives
\[
B_A(T)\ll_{\gamma,\epsilon}\lambda^{-1/2}\left(A^{2+\epsilon}+\frac{A^{4+\epsilon}}{T}\right)\ll_{\gamma,\epsilon}\lambda^{-1/2}\frac{A^{4+\epsilon}}{T}.
\]
On the dual side, we have
\[
B_K^{(2)}(T)\ll_\gamma K+\frac{K^2}{T}.
\]
Therefore, the total upper bound is
\begin{align*}
|S|&\ll_{\gamma,\epsilon}T^{1/2}\left(\lambda^{-1/2}\frac{A^{4+\epsilon}}{T}\right)^{1/2}\left(K+\frac{K^2}{T}\right)^{1/2}  \\
&\ll_{\gamma,\epsilon}\lambda^{-1/4}A^{2+\epsilon}\left(K^{1/2}+KT^{-1/2}\right) \\
&=\lambda^{-1/4}A^{2+\epsilon}K\left(K^{-1/2}+T^{-1/2}\right).
\end{align*}

We next suppose that $T>A^2$. In this range the estimate for $B_A(T)$ is much smaller, but a direct application of the double large sieve does not give the desired first term. We first apply Cauchy--Schwarz in the $(a,b)$-variables:

\begin{align*}
|S|^2&\le O(A^2)\sum_{a,b\asymp A}\left|\sum_{k\asymp K}d_k e\!\left(Tu(a,b)v(k)\right)\right|^2 \\
&=O(A^2)\sum_{a,b\asymp A}\sum_{k_1,k_2\asymp K}d_{k_1}\overline{d_{k_2}}\,e\!\left(Tu(a,b)\bigl(v(k_1)-v(k_2)\bigr)\right).
\end{align*}

We now apply Lemma \ref{lem:DLS} to the two families $x_{a,b}=u(a,b)$ and $y_{k_1,k_2}=v(k_1)-v(k_2)$. The pair-spacing count for the first family is $B_A(T)$. For the second family it is
\[
\#\left\{k_1,k_2,k_3,k_4\asymp K:\left|v(k_1)-v(k_2)-v(k_3)+v(k_4)\right|\le T^{-1}\right\},
\]
which is exactly $B_K^{(4)}(T)$. Hence we obtain
\[
|S|^2\ll A^2T^{1/2}B_A(T)^{1/2}B_K^{(4)}(T)^{1/2}.
\]

Since $T>A^2$, Lemma~\ref{lem:mixedspacing} now gives
\[
B_A(T)\ll_{\gamma,\epsilon}\lambda^{-1/2}A^{2+\epsilon},
\]
while
\[
B_K^{(4)}(T)\ll_{\gamma,\epsilon}K^{2+\epsilon}+\frac{K^{4+\epsilon}}{T}.
\]
Substituting these estimates, we can deduce that
\[
\begin{aligned}
|S|^2&\ll_{\gamma,\epsilon}\lambda^{-1/4}A^{3+\epsilon}T^{1/2}\left(K^{2+\epsilon}+\frac{K^{4+\epsilon}}{T}\right)^{1/2} \\
&\ll_{\gamma,\epsilon}\lambda^{-1/4}(AK)^{O(\epsilon)}\left(A^3KT^{1/2}+A^3K^2\right).
\end{aligned}
\]
Taking square roots gives
\[
|S|\ll_{\gamma,\epsilon}\lambda^{-1/8}(AK)^\epsilon\left(A^{3/2}K^{1/2}T^{1/4}+A^{3/2}K\right).
\]

Since $\lambda\le1$, we have $\lambda^{-1/8}\le\lambda^{-1/4}$. The two estimates above therefore combine to give
\[
|S|\ll_{\gamma,\epsilon}\lambda^{-1/4}(A^2K)^{1+\epsilon}\left\{\left(\frac{T}{A^2K^2}\right)^{1/4}+A^{-1/2}+K^{-1/2}+T^{-1/2}\right\},
\]
after renaming $\epsilon$. This proves the proposition.
\end{proof}

\subsubsection{Two nonzero frequencies}

We now return to the sum $\mathcal{B}$ introduced in Section~4.2. Consider the part for which at least one phase lies in the large-$W$ range and the two Fourier frequencies in that phase are both nonzero. More precisely, put
\[
\mathcal{B}_{\mathrm{large}}^{(2)}:=\sum_{\mathbf{r}\in\mathcal{F}}|c(\mathbf{r})|\sum_{\substack{a\sim A_1,\ b\sim A_2,\ a\ne b\\\exists\,j\in\{1,2\}:\ W_j>G_j,\ r_js_j\ne0}}|S_1(C_1)S_2(C_2)|.
\]
We shall show that this part already has a power saving.

\begin{proposition}\label{prop:highnonzero}
For every fixed balanced block, we have
\[
\mathcal{B}_{\mathrm{large}}^{(2)}\ll_{\gamma,\epsilon}Y^{7/4+2\epsilon+o(1)}.
\]
\end{proposition}

\begin{proof}
We first work at comparable scales
\[
a,b\asymp A,\qquad g,h\asymp G,\qquad AG=Y.
\]
The proportion defined in balanced scales costs only $\mathcal{R}(Y)^{O_I(1)}=Y^{o(1)}$.

Fix the selected nonzero frequency pair $(r,s)$ and a dyadic range
\[
C_0\le |C(a,b)|<2C_0,\qquad C(a,b)=ra^\gamma+sb^\gamma,
\]
and put $W=C_0G^\gamma$. By symmetry, suppose that $C_1$ is the selected large phase. We apply the $B$-process to $S_1(C_1)$ and estimate the other factor trivially:
\[
|S_2(C_2)|\ll G\,Y^{o(1)}.
\]
Thus we can define
\begin{equation}\label{eq:SrsW}
\mathcal{S}_{r,s}(W):=\sum_{\substack{a,b\asymp A\\C_0\le |C_1(a,b)|<2C_0}}|S_1(C_1(a,b))|.
\end{equation}
It remains to estimate
\[
GY^{o(1)}\mathcal{S}_{r,s}(W).
\]

By Lemma~\ref{lem:Bprocess}, the $B$-process transforms $S_1(C_1(a,b))$ into a dual sum of length
\[
K\asymp_\gamma\frac{W}{G}
\]
with the factor $G/\sqrt W$. After the normalization introduced above, the dual oscillation has the form
\[
e\!\left(\sigma T u(a,b)v(k)\right),\qquad T\asymp_\gamma W.
\]
Hence Proposition~\ref{prop:transfer} gives
\[
\begin{aligned}
\mathcal{S}_{r,s}(W)\ll_{\gamma,\epsilon}\lambda^{-1/4}Y^{o(1)}\frac{G}{\sqrt W}A^2K\bigg\{\left(\frac{T}{A^2K^2}\right)^{1/4}+A^{-1/2}+K^{-1/2}+T^{-1/2}\bigg\},
\end{aligned}
\]
with $\lambda$ as defined above. Here the factor $(A^2K)^\epsilon$ from Proposition~\ref{prop:transfer} has been absorbed into $Y^{o(1)}$.

Using the preceding relations for $K$ and $T$, we obtain
\[
\mathcal{S}_{r,s}(W)\ll_{\gamma,\epsilon}\lambda^{-1/4}Y^{o(1)}\left(A^{3/2}G^{1/2}W^{1/4}+A^{3/2}W^{1/2}+A^2G^{1/2}+A^2\right).
\]
Therefore, it follows that
\[
GY^{o(1)}\mathcal{S}_{r,s}(W)\ll_{\gamma,\epsilon}\lambda^{-1/4}Y^{o(1)}\left(A^{3/2}G^{3/2}W^{1/4}+A^{3/2}GW^{1/2}+A^2G^{3/2}+A^2G\right).
\]

By the Fourier cutoff $Q\ll HY^\epsilon$ and $H=Y^{1-\gamma}$, we know that
\[
W\ll QY^\gamma\mathcal R(Y)^{O_I(1)}\ll Y^{1+\epsilon+o(1)}.
\]
Using $AG=Y$ and $A\le Y^{1/2}$, the preceding four terms are therefore bounded by
\[
A^{3/2}G^{3/2}W^{1/4}+A^{3/2}GW^{1/2}+A^2G^{3/2}+A^2G\ll Y^{7/4+\epsilon/2+o(1)}.
\]
Thus, for the fixed dyadic frequency range,
\[
GY^{o(1)}\mathcal{S}_{r,s}(W)\ll_{\gamma,\epsilon}\lambda^{-1/4}Y^{7/4+\epsilon/2+o(1)}.
\]

It remains only to sum the Fourier coefficients. On $|r|\asymp Q$ and $|s|\asymp\lambda Q$, their total mass is
\[
\ll H^{-2}\lambda Q^2\left(1+\frac{Q}{H}\right)^{-B}.
\]
Hence the factor $\lambda^{-1/4}$ above leaves $\lambda^{3/4}$, and for $B>3$,
\[
\sum_{\lambda\ {\rm dyadic}}\lambda^{3/4}\ll1,\qquad H^{-2}\sum_{Q\ {\rm dyadic}}Q^2\left(1+\frac{Q}{H}\right)^{-B}\ll1.
\]
The second frequency pair has bounded total Fourier mass, while the dyadic subdivision in $W$ and the balanced-scale adjustments contribute only $Y^{o(1)}$.

Finally, the $B$-process error contributes
\[
A^2G\left(\frac{G}{\sqrt W}+1\right)Y^{o(1)}\ll Y^{7/4+o(1)}+Y^{3/2+o(1)},
\]
since $W>G$ and $A\le Y^{1/2}$. Therefore, it follows that
\[
\mathcal{B}_{\mathrm{large}}^{(2)}\ll_{\gamma,\epsilon}Y^{7/4+2\epsilon+o(1)},
\]
which proves the proposition.
\end{proof}

\subsubsection{Exactly one nonzero frequency}

The preceding proposition does not apply when the selected phase has the form
\[
C(a,b)=ra^\gamma,\qquad r\ne0,
\]
or symmetrically $C(a,b)=sb^\gamma$. The zero Fourier frequency now provides an additional factor $H^{-1}$.

\begin{lemma}\label{lem:onedtransfer}
Let
\[
u(a)=(a/A)^\beta,\qquad v(k)=(K/|k|)^\beta.
\]
If $T\gg_\gamma A$ and $A,K\ge1$, then for arbitrary bounded coefficients we have
\[
\left|\sum_{a\asymp A}c_a\sum_{|k|\asymp K}d_k e(Tu(a)v(k))\right|\ll_{\gamma,\epsilon}(AK)^\epsilon\left(A^{3/4}K^{1/2}T^{1/4}+A^{3/4}K\right).
\]
\end{lemma}

\begin{proof}
The monotonicity gives
\[
\#\{a,a':|u(a)-u(a')|\le T^{-1}\}\ll A+\frac{A^2}{T}\ll A.
\]
Apply Cauchy--Schwarz inequality in $a$ and then the double large sieve to $u(a)$ and $v(k_1)-v(k_2)$ like Proposition \ref{prop:transfer}. Using the four point estimate on the dual side gives the result.
\end{proof}

Let $\mathcal{B}_{\mathrm{large}}^{(1)}$ denote the part of $\mathcal{B}$ for which at least one phase lies in the large-$W$ range and exactly one Fourier frequency in that phase is nonzero.

\begin{proposition}\label{prop:highsingle}
For every fixed balanced block, we have
\[
\mathcal{B}_{\mathrm{large}}^{(1)}\ll_{\gamma,\epsilon}Y^{\gamma+7/8+2\epsilon+o(1)}.
\]
\end{proposition}

\begin{proof}
Suppose the selected phase is $ra^\gamma$, $r\ne0$. Put $T_0:=|r|Y^\gamma$ and $K:=\frac{T_0}{G}$. The large-$W$ condition gives $T_0\gg G\ge A$, apart from a bounded transition range which is covered by the second-derivative estimate.

As in the preceding subsection, we write the transformed phase in the product form
\[
e\!\left(\sigma T u(a)v(k)\right),
\]
where now
\[
u(a)=\left(\frac{a}{A}\right)^\beta,\qquad v(k)=\left(\frac{K}{|k|}\right)^\beta,\qquad T\asymp_\gamma T_0.
\]
Hence Lemma~\ref{lem:onedtransfer} applies. Restoring the $B$-process factor $G/\sqrt{T_0}$ gives
\[
\sum_{a\asymp A}\left|\sum_{g\asymp G}w(g/G)e(ra^\gamma g^\gamma)\right|\ll_{\gamma,\epsilon}Y^{o(1)}\left(A^{3/4}G^{1/2}T_0^{1/4}+A^{3/4}T_0^{1/2}\right).
\]

The remaining $b$-sum and the other long sum contribute $AG$ trivially. The selected frequency pair has weight
\[
\ll H^{-2}\left(1+\frac{|r|}{H}\right)^{-B}.
\]
Using
\[
\sum_{r\ge1}r^{1/4}\left(1+\frac rH\right)^{-B}\ll H^{5/4},
\qquad
\sum_{r\ge1}r^{1/2}\left(1+\frac rH\right)^{-B}\ll H^{3/2},
\]
the contribution of the two main terms above is
\[
\begin{aligned}
\mathcal{B}_{\mathrm{large,main}}^{(1)}&\ll Y^{o(1)}\left(A^{7/4}G^{3/2}H^{-3/4}Y^{\gamma/4}+A^{7/4}GH^{-1/2}Y^{\gamma/2}\right) \\
&=Y^{o(1)}\left(A^{1/4}Y^{\gamma+3/4}+A^{3/4}Y^{\gamma+1/2}\right)\ll Y^{\gamma+7/8+o(1)}.
\end{aligned}
\]
Here we used $AG=Y$, $H=Y^{1-\gamma}$, and $A\le Y^{1/2}$.

The $B$-process remainder contributes
\[
\mathcal{B}_{\mathrm{large,err}}^{(1)}\ll Y^{\gamma+1/2+o(1)}.
\]
The balanced-scale distortions are absorbed into $Y^{o(1)}$. Hence
\[
\mathcal{B}_{\mathrm{large}}^{(1)}\ll_{\gamma,\epsilon}Y^{\gamma+7/8+2\epsilon+o(1)},
\]
which proves the proposition.
\end{proof}

Combining Propositions~\ref{prop:highnonzero} and \ref{prop:highsingle}, the contribution of all configurations
containing a large-$W$ phase is
\[
\ll_{\gamma,\epsilon}Y^{7/4+2\epsilon+o(1)}+Y^{\gamma+7/8+2\epsilon+o(1)}=o(Y^{2\gamma}),
\]
provided $\epsilon>0$ is sufficiently small, since $\gamma\ge\gamma_0>7/8$.

\subsubsection{Zero frequencies and completion of Theorem~\ref{thm:energy}}

We now combine the three ranges. Define the total exceptional set as
\[
\mathcal{E}_I:=\bigl(\mathcal{E}_{\rm exact}\cap I\bigr)\cup\limsup_{j\to\infty} E_{2^j}.
\]
The first set is countable, while the second has measure zero by the Borel--Cantelli argument in Section~4.4. Hence
\[
\meas(\mathcal E_I)=0.
\]

Fix $\gamma\in I\setminus\mathcal{E}_I$. Then, for all sufficiently large dyadic $Y$, no exact cancellation occurs, and no small-$A$ near-cancellation event belongs to $E_Y$. Now we know every nonzero phase belongs to exactly one of the following ranges
\[
0<W_j\le Y^\eta,\qquad Y^\eta<W_j\le G_j,\qquad W_j>G_j.
\]

The small-$W$ range was treated in Section~4.4, while the large-$W$ range was treated in Section~4.5. If both frequency pairs are nonzero and neither range occurs, then both phases lie in the second-derivative range, whose contribution can be bounded in Section~4.3 as
\[
\ll Y^{2-\eta+o(1)}=o(Y^{2\gamma}).
\]

Therefore, it only remains to consider configurations involving zero Fourier frequencies. Suppose first that
\[
(r_1,s_1)=(0,0),\qquad (r_2,s_2)\ne(0,0),
\]
or symmetrically with the two pairs interchanged. Then $C_1=0$, so the first long sum is estimated trivially, while the second phase is assumed to lie in the second-derivative range
\[
Y^\eta<W_2\le G_2.
\]
The zero frequency pair contributes $H^{-2}Y^{o(1)}$ through its two Fourier coefficients, and the second long sum saves a factor $Y^{-\eta/2}$. Hence the total contribution is
\[
\ll A^2G^2H^{-2}Y^{-\eta/2+o(1)}=Y^{2\gamma-\eta/2+o(1)}=o(Y^{2\gamma}).
\]

If both frequency pairs vanish, that is,
\[
(r_1,s_1)=(r_2,s_2)=(0,0),
\]
then both phases are zero. The four zero-frequency Fourier coefficients contribute $H^{-4}Y^{o(1)}$, and both long sums are
estimated trivially. Thus, we obtain
\[
\ll A^2G^2H^{-4}Y^{o(1)}=Y^{4\gamma-2+o(1)}=o(Y^{2\gamma}),
\]
since $\gamma<1$.

Choose $\epsilon>0$ sufficiently small so that
\[
10\epsilon<\min\left\{\kappa,\,\eta-2s_0,\,1-6s_0-2\kappa-\eta,\,\gamma_0-\frac78,\,1-\gamma_1\right\}.
\]
Then every contribution above has a positive power saving below $Y^{2\gamma}$. Thus, for fixed $\gamma$ outside the exceptional set, the total off-diagonal contribution at scale $Y$ is
\[
\omega_\gamma(Y)Y^{2\gamma},\qquad\omega_\gamma(Y)\longrightarrow0.
\]

\begin{proof}[Proof of Theorem~\ref{thm:energy}]
Lemma~\ref{lem:extreme} treats the unbalanced blocks. The preceding arguments treat all balanced blocks outside a null subset of each compact interval $I\subset(7/8,1)$. It remains only to sum over dyadic $Y$. Given $\delta>0$, choose $Y_0$ so that
\[
\omega_\gamma(Y)\le\delta\qquad (Y\ge Y_0).
\]
Then it follows that
\[
\sum_{\substack{Y\ll X\\Y\ {\rm dyadic}}}\omega_\gamma(Y)Y^{2\gamma}\le O_{\gamma,Y_0}(1)+C_\gamma\delta X^{2\gamma}.
\]
First letting $X\to\infty$ and then $\delta\to0$ gives
\[
N_{\rm off}\!\left(\mathcal{N}^{(1/\gamma)}(X)\right)=o(X^{2\gamma}).
\]
Since
\[
|\mathcal{N}^{(1/\gamma)}(X)|\sim X^\gamma,
\]
Lemma~\ref{lem:rectangles} yields
\[
E^\times(\mathcal{N}^{(1/\gamma)}(X))=2|\mathcal{N}^{(1/\gamma)}(X)|^2-|\mathcal{N}^{(1/\gamma)}(X)|+o\!\left(|\mathcal{N}^{(1/\gamma)}(X)|^2\right).
\]

Finally cover $(7/8,1)$ by countably many compact intervals. The union of the corresponding exceptional sets has measure zero. Since $c=1/\gamma$ is bi-Lipschitz on each compact interval, null sets are preserved under this change of variables. This proves Theorem~\ref{thm:energy}.
\end{proof}

\subsection{Proof of Theorem~\ref{thm:gaussian}}\label{sec:CLT}
Since $|\mathcal{N}^{(c)}(X)|\asymp X^\gamma$ with $\gamma<1$, the simplified criterion of Soundararajan and Xu \cite[Theorem~1.1]{SoundXu} does not apply directly. We therefore use their more general criterion \cite[Corollary~3.2]{SoundXu}, which we record in the form needed below.

\begin{lemma}[Soundararajan--Xu]\label{lem:SXcriterion}
Let $\mathcal{A}$ be a nonempty finite set of positive integers and let $f$ be a Steinhaus random multiplicative function. Define the complex-valued random variable
\[
Z=\frac{1}{\sqrt{|\mathcal A|}}\sum_{n\in\mathcal A}f(n).
\]
Let $\mathcal{S}$ be a subset of $\mathcal{A}$, with all elements in $\mathcal{S}$ being at least $2$. Suppose that $1\ge\varepsilon\ge0$ and that the following three conditions are satisfied:

\noindent
{\rm (1).}\quad$|\mathcal{A}\setminus\mathcal{S}|\le \varepsilon^2|\mathcal{A}|$.

\noindent
{\rm (2).}\quad The number of solutions to $m_1m_2=n_1n_2$ with $m_1,m_2,n_1,n_2\in\mathcal{S}$ and $m_1\ne n_1$, $m_1\ne n_2$ is at most $\varepsilon^4|\mathcal{A}|^2$.

\noindent
{\rm (3).}\quad For each prime $\ell$,
\[
\#\{n\in\mathcal{S}:P^+(n)=\ell\}\le \varepsilon^4|\mathcal{A}|.
\]

Then, for any real numbers $t_1$ and $t_2$ with $t^2=\frac{t_1^2+t_2^2}{2}$, we have
\[
\mathbb{E}e^{\,it_1\operatorname{Re}(Z)+it_2\operatorname{Im}(Z)}=e^{-t^2/2}+O\!\left(e^{t^2}\varepsilon\right).
\]
In particular, if $\varepsilon\to0$, then
\[
Z\xrightarrow{d}\mathcal{CN}(0,1).
\]
\end{lemma}

For the application, we take
\[
\mathcal{A}=\mathcal{N}^{(c)}(X),\qquad\mathcal{S}=\mathcal{A}\setminus\{1\}.
\]
Thus $\mathcal{S}$ differs from $\mathcal{A}$ by only one element, so condition {\rm (1)} is immediate. Condition {\rm (2)} follows from Theorem~\ref{thm:energy}, since deleting $1$ cannot create new off-diagonal multiplicative relations. Hence it remains only to verify condition {\rm (3)}.

\begin{lemma}\label{lem:largestprime}
For fixed $\gamma\in(7/8,1)$, put $c=1/\gamma$ and $M:=|\mathcal N^{(c)}(X)|$. Then we have
\[
\sup_{\ell\ {\rm prime}}\#\{m\in\mathcal N^{(c)}(X):P^+(m)=\ell\}=o_\gamma(M).
\]
\end{lemma}

\begin{proof}
Put $z:=(\log X)^4$. We divide the primes $\ell$ into three ranges.

Suppose first that $\ell\le z$. Then every $m$ with $P^+(m)=\ell$ is $z$-smooth. Writing
\[
\Psi_{\rm smooth}(X,z):=\#\{m\le X:P^+(m)\le z\},
\]
Rankin's trick with $\sigma=4/5$ gives
\[
\Psi_{\rm smooth}(X,z)\le X^\sigma\sum_{\substack{m\ge1\\ P^+(m)\le z}}m^{-\sigma}=X^\sigma\prod_{q\le z}(1-q^{-\sigma})^{-1}.
\]
Moreover,
\[
\log\prod_{q\le z}(1-q^{-\sigma})^{-1}\ll\sum_{n\le z}n^{-\sigma}\ll z^{1-\sigma}=(\log X)^{4/5}.
\]
Hence, uniformly for $\ell\le z$,
\[
\#\{m\in\mathcal{N}^{(c)}(X):P^+(m)=\ell\}\le\Psi_{\rm smooth}(X,z)\le X^{4/5+o(1)}=o(X^\gamma)=o(M).
\]

Now suppose that $z<\ell\le X^{1-\gamma}z$. We drop the condition $P^+(m)=\ell$ and simply count multiples of $\ell$ in $\mathcal{N}^{(c)}(X)$. Since $n\mapsto\lfloor n^c\rfloor$ is strictly increasing, this number is at most
\[
\#\{n\le M:\ell\mid\lfloor n^c\rfloor\}.
\]
The divisibility condition $\ell\mid\lfloor n^c\rfloor$ is equivalent to
\[
\left\{\frac{n^c}{\ell}\right\}<\frac{1}{\ell}.
\]

Applying the Erd\H{o}s--Tur\'an inequality \cite{KuipersNiederreiter} with cutoff $\ell$, we get
\[
\#\{n\le M:\ell\mid\lfloor n^c\rfloor\}\ll\frac{M}{\ell}+\sum_{j=1}^{\ell}\frac1j\left|\sum_{n\le M}e(jn^c/\ell)\right|.
\]
On dyadic subintervals, Lemma~\ref{lem:second} gives
\[
\sum_{n\le M}e(jn^c/\ell)\ll_c\left(\frac j\ell\right)^{1/2}M^{c/2}+\left(\frac\ell j\right)^{1/2}M^{1-c/2}.
\]
Since $1<c<2$, the dyadic pieces sum geometrically. Using
\[
\sum_{j\le\ell}j^{-1/2}\ll \ell^{1/2},\qquad\sum_{j\le\ell}j^{-3/2}\ll1,
\]
we obtain
\[
\#\{n\le M:\ell\mid\lfloor n^c\rfloor\}\ll_c\frac{M}{\ell}+M^{c/2}+\ell^{1/2}M^{1-c/2}.
\]
Since $M\asymp X^\gamma$ and $\ell\le X^{1-\gamma}z$, the right-hand side is
\[
\ll_c\frac{M}{z}+X^{1/2}+X^{\gamma/2}(\log X)^2=o(M),
\]
because $\gamma>7/8$.

Finally, suppose that $\ell>X^{1-\gamma}z$. Then we have
\[
\#\{m\in\mathcal{N}^{(c)}(X):P^+(m)=\ell\}\le\#\{m\le X:\ell\mid m\}\le\frac{X}{\ell}\le\frac{X^\gamma}{z}=o(M).
\]

Thus, uniformly over all primes $\ell$, we can conclude that
\[
\#\{m\in\mathcal{N}^{(c)}(X):P^+(m)=\ell\}=o_\gamma(M),
\]
which proves the lemma.
\end{proof}

Therefore, we can now finish the last step of Theorem~\ref{thm:gaussian}. Lemma~\ref{lem:largestprime} verifies condition {\rm (3)} of Lemma~\ref{lem:SXcriterion}. Together with the observations preceding Lemma~\ref{lem:largestprime}, all three conditions are therefore satisfied asymptotically.

Indeed, with $M=|\mathcal{A}|$, choose
\[
\varepsilon_X:=\max\left\{M^{-1/2},\left(\frac{N_{\rm off}(\mathcal S)}{M^2}\right)^{1/4},\left(\frac{\sup_{\ell\ {\rm prime}}\#\{n\in\mathcal{S}:P^+(n)=\ell\}}{M}\right)^{1/4}\right\}.
\]
Then $\varepsilon_X\to0$, and conditions {\rm (1)}--{\rm (3)} of Lemma~\ref{lem:SXcriterion} hold with $\varepsilon=\varepsilon_X$. Hence we conclude that
\[
\frac{S_c(X;f)}{\sqrt M}\xrightarrow{d}\mathcal{CN}(0,1).
\]

It remains to prove the first-moment asymptotic. Put $Z_X:=\frac{S_c(X;f)}{\sqrt M}$. By Steinhaus orthogonality, it follows that $\mathbb{E}|Z_X|^2=1$. Hence
\[
\mathbb{E}\!\left(|Z_X|\mathbf{1}_{\{|Z_X|>R\}}\right)\le R^{-1},\qquad R>0,
\]
so $(|Z_X|)$ is uniformly integrable. Since
\[
Z_X\xrightarrow{d}Z,\qquad Z\sim\mathcal{CN}(0,1),
\]
we obtain
\[
\mathbb{E}|Z_X|\longrightarrow\mathbb{E}|Z|=\frac{\sqrt\pi}{2}.
\]
Therefore,
\[
\mathbb{E}|S_c(X;f)|\sim\frac{\sqrt\pi}{2}\sqrt M.
\]
This completes the proof of Theorem~\ref{thm:gaussian}.

\section{Exceptional exponents}\label{sec:exceptional}

We now prove Proposition~\ref{prop:exceptional}. The construction is based on an exact dilation relation and shows that the exceptional set in Theorem~\ref{thm:energy} has $1$ as an accumulation point.

\begin{proof}[Proof of Proposition~\ref{prop:exceptional}]
Fix an integer $u\ge5$ and put $v:=u+1,c:=c_u=\frac{\log v}{\log u}$. Then we have $u^c=v$ immediately. Moreover, it follows that
\[
c_u=1+\frac{\log(1+1/u)}{\log u}\longrightarrow1\qquad (u\to\infty).
\]

Since  $u$ is at least 5, the following elementary inequality
\[
\left(1+\frac1u\right)^7<u
\]
gives $1<c_u<\frac{8}{7}$.

Therefore, we can define $M:=|\mathcal{N}^{(c)}(X)|=X^\gamma+O(1)$ and $N:=\left\lfloor\frac{X^\gamma}{u}\right\rfloor$. For $1\le n\le N$, write $m_n:=\lfloor n^c\rfloor$. If $\{n^c\}<\frac{1}{v}$, then it follows that
\[
\lfloor(un)^c\rfloor=\lfloor vn^c\rfloor=v\lfloor n^c\rfloor=vm_n.
\]
Also $un\le X^\gamma$, so both $m_n$ and $vm_n$ belong to $\mathcal{N}^{(c)}(X)$.

We next count such integers $n$. For every fixed nonzero integer $h$, Lemma~\ref{lem:second} and dyadic summation give
\[
\sum_{n\le N}e(hn^c)\ll_c|h|^{1/2}N^{c/2}+|h|^{-1/2}N^{1-c/2}=o_h(N),
\]
because $1<c<2$. Hence, by Weyl's criterion \cite[Chapter~1]{KuipersNiederreiter}, the sequence $(n^c)$ is uniformly distributed modulo one, which means that
\[
R:=\#\left\{1\le n\le N:\{n^c\}<\frac{1}{v}\right\}=\frac{N}{v}+o(N).
\]
Recalling the definition of $N$ we know that
\[
N=\frac{X^\gamma}{u}+O(1)=\left(\frac{1}{u}+o(1)\right)M.
\]
Thus we obtain
\[
R=\left(\frac{1}{uv}+o(1)\right)M.
\]

Let $\mathcal{G}$ be the corresponding set of values $m_n$. The map $n\mapsto\lfloor n^c\rfloor$ is strictly increasing, so the elements of $\mathcal{G}$ are distinct. For any distinct ordered pair $g,h\in\mathcal{G}$, all four integers
\[
g,\qquad vg,\qquad h,\qquad vh
\]
belong to $\mathcal{N}^{(c)}(X)$, and $(vg)h=g(vh)$. In the parametrization of Lemma~\ref{lem:rectangles}, the choice $(a,b)=(v,1)$ therefore produces $R(R-1)$ non-diagonal solutions. The reversed choice $(a,b)=(1,v)$ gives another disjoint family of the same size. Hence the multiplicative energy satisfies
\[
\begin{aligned}
E^\times(\mathcal{N}^{(c)}(X))&\ge2M^2-M+2R(R-1) \\
&=\left(2+\frac{2}{u^2v^2}+o(1)\right)M^2.
\end{aligned}
\]
Since $v=u+1$,
\[
\liminf_{X\to\infty}\frac{E^\times(\mathcal{N}^{(c_u)}(X))}{|\mathcal{N}^{(c_u)}(X)|^2}\ge2+\frac{2}{u^2(u+1)^2}>2.
\]
This proves the proposition.
\end{proof}

Thus the exceptional set in Theorem~\ref{thm:energy} is not only nonempty; it has $1$ as an accumulation point. Also, it is worth noting that the existence of exceptional set does not contradict Theorem~\ref{thm:near}. For every fixed $u$, the number $c_u-1$ is positive, and therefore
\[
(c_u-1)\log X\longrightarrow\infty,
\]
whereas Theorem~\ref{thm:near} assumes
\[
(c-1)\log X\longrightarrow0.
\]

\section{Conclusion}

We have studied two different regimes of random multiplicative sums over Piatetski-Shapiro sequences. When $c=c(X)$ approaches $1$ sufficiently fast, $\mathcal{N}^{(c)}(X)$ differs from the interval $[1,X]$ by only a small set. This allows Harper's
better than square-root cancellation estimate to be transferred to the Piatetski-Shapiro setting, and the same comparison also gives the character-sum consequence. For almost every fixed $1<c<8/7$, on the other hand, we prove that
\[
E^\times(\mathcal{N}^{(c)}(X))=2|\mathcal{N}^{(c)}(X)|^2-|\mathcal{N}^{(c)}(X)|+o\!\left(|\mathcal{N}^{(c)}(X)|^2\right).
\]
Together with the largest prime factor estimate, the criterion of Soundararajan and Xu then yields a standard complex Gaussian limit for the normalized Steinhaus sum.

The main difficulty in the fixed-exponent case is the treatment of the Fourier phases. The second-derivative estimate handles the intermediate range. Small phases are controlled by a metric argument for small $A$ and by deterministic counting for large $A$, while large phases are transformed by the Van der Corput's $B$-process and estimated using the double large sieve together with the Robert-Sargos spacing theorem. Exact cancellations and configurations involving zero Fourier frequencies are treated separately.

The almost everywhere restriction in the energy theorem is genuine: the explicit family
\[
c_u=\frac{\log(u+1)}{\log u}
\]
gives exceptional fixed exponents tending to $1$. This does not conflict with the moving-exponent result, since each $c_u>1$ is fixed. It remains open to understand the transition between the two regimes. In particular, our arguments do not cover the intermediate range in which $(c-1)\log X$ tends to a positive constant or grows slowly.

\section*{Acknowledgements}

The authors are grateful to Professors Yuan Yi, Victor Zhenyu Guo, and Adam J. Harper for carefully reading an earlier version of this work, pointing out several issues, and providing valuable comments and suggestions that led to a substantial revision of the paper.

This work is supported by the National Natural Science Foundation of China (Nos.~12671010 and 12271320), the Natural Science Basic Research Plan in Shaanxi Province (No.~2026JC-YBMS-0094), and the Shaanxi Fundamental Science Research Project for Mathematics and Physics (No.~25JSY029).



\begingroup
\small
\begin{thebibliography}{99}

\bibitem{BakerBanks}
R. C. Baker and W. D. Banks, \emph{Character sums with Piatetski-Shapiro sequences}, Q. J. Math. \textbf{66} (2015), no.~2, 393--416.

\bibitem{BanksShparlinskiChar}
W. D. Banks and I. E. Shparlinski, \emph{Multiplicative character sums with twice-differentiable functions}, Q. J. Math. \textbf{60} (2009), no.~4, 401--411.

\bibitem{HarmanMetric}
G. Harman, \emph{Metric Number Theory}, London Math. Soc. Monogr. New Ser. 18, Clarendon Press, Oxford University Press, Oxford, 1998.

\bibitem{BanksShparlinskiBeatty}
W. D. Banks and I. E. Shparlinski, \emph{Prime numbers with Beatty sequences}, Colloq. Math. \textbf{115} (2009), no.~2, 147--157.

\bibitem{BombieriIwaniec}
E. Bombieri and H. Iwaniec, \emph{Some mean-value theorems for exponential sums}, Ann. Scuola Norm. Sup. Pisa Cl. Sci. (4) \textbf{13} (1986), no.~3, 473--486.

\bibitem{GrahamKolesnik}
S. W. Graham and G. Kolesnik, \emph{Van der Corput's Method of Exponential Sums}, London Math. Soc. Lecture Note Ser. 126, Cambridge University Press, Cambridge, 1991.

\bibitem{GuoBeatty}
V. Z. Guo, \emph{Piatetski-Shapiro primes in a Beatty sequence}, J. Number Theory \textbf{156} (2015), 317--330.

\bibitem{HardyXu}
S. Hardy and M. W. Xu, \emph{Helson's conjecture for smooth numbers}, arXiv:2511.03430, 2025.

\bibitem{HarperRMF}
A. J. Harper, \emph{Moments of random multiplicative functions, I: Low moments, better than squareroot cancellation, and critical multiplicative chaos}, Forum Math. Pi \textbf{8} (2020), e1, 95 pp.

\bibitem{HarperChar}
A. J. Harper, \emph{The typical size of character and zeta sums is $o(\sqrt{x})$}, arXiv:2301.04390, 2023.

\bibitem{HarperShort}
A. J. Harper, \emph{A note on character sums over short moving intervals}, J. Inst. Math. Jussieu \textbf{24} (2025), no.~4, 1395--1427.

\bibitem{Helson}
H. Helson, \emph{Hankel forms}, Studia Math. \textbf{198} (2010), no.~1, 79--84.

\bibitem{IwaniecKowalski}
H. Iwaniec and E. Kowalski, \emph{Analytic Number Theory}, Amer. Math. Soc. Colloq. Publ. 53, American Mathematical Society, Providence, RI, 2004.

\bibitem{KuipersNiederreiter}
L. Kuipers and H. Niederreiter, \emph{Uniform Distribution of Sequences}, Wiley-Interscience, New York, 1974.

\bibitem{Kumchev}
A. Kumchev, \emph{On the Piatetski-Shapiro-Vinogradov theorem}, J. Th\'eor. Nombres Bordeaux \textbf{9} (1997), no.~1, 11--23.

\bibitem{LuGuoGuo}
L. Lu, L. Y. Guo, and V. Z. Guo, \emph{Improvements on exponential sums related to Piatetski-Shapiro primes}, J. Number Theory \textbf{281} (2026), 700--725.

\bibitem{PiatetskiShapiro}
I. I. Piatetski-Shapiro, \emph{On the distribution of prime numbers in sequences of the form $[f(n)]$}, Mat. Sb. (N.S.) \textbf{33}(75) (1953), no.~3, 559--566.

\bibitem{RobertSargos}
O. Robert and P. Sargos, \emph{Three-dimensional exponential sums with monomials}, J. Reine Angew. Math. \textbf{591} (2006), 1--20.

\bibitem{SoundXu}
K. Soundararajan and M. W. Xu, \emph{Central limit theorems for random multiplicative functions}, J. Anal. Math. \textbf{151} (2023), no.~1, 343--374.

\bibitem{WangXuBeyond}
V. Y. Wang and M. W. Xu, \emph{Harper's beyond square-root conjecture}, Int. Math. Res. Not. IMRN \textbf{2025} (2025), no.~18, rnaf279.

\bibitem{XuRough}
M. W. Xu, \emph{Better than square-root cancellation for random multiplicative functions}, Trans. Amer. Math. Soc. Ser. B \textbf{11} (2024), no.~14, 482--507.

\end{thebibliography}
\endgroup
\end{document}